\documentclass[12pt,draftcls,onecolumn]{IEEEtran}

\ifCLASSINFOpdf
  \usepackage[pdftex]{graphicx}
\else
  \usepackage[dvips]{graphicx}
\fi
\usepackage{amsmath}

\usepackage{array}

\usepackage{amsthm}
\usepackage{amssymb}
\usepackage{xcolor}
\usepackage{algpseudocode,algorithm,algorithmicx}
\usepackage{subcaption}
\usepackage{easyReview}
\DeclareCaptionLabelSeparator{periodspace}{.\quad}
\usepackage[export]{adjustbox}
\newtheorem{thm}{Theorem}
\newtheorem{lem}[thm]{Lemma}

\newtheorem{cor}[thm]{Corollary}
\theoremstyle{definition}
\newtheorem{df}{Definition}
\newtheorem{rem}{Remark}
\newtheorem{ex}{Example}

\def\<{\langle}
\def\>{\rangle}

\def \w {\omega}

\def \b {\beta}
\def \s {\sigma}

\def \e {\varepsilon}

\begin{document}
%

\title{A Convex-Geometric Approach to Ensemble Control Analysis and Design in a Hilbert Space}

%
%

\author{Wei~Miao,~\IEEEmembership{Student Member,~IEEE,}
        and~Jr-Shin~Li,~\IEEEmembership{Senior~Member,~IEEE}
\thanks{*This work was supported in part by the National Science Foundation under the awards ECCS-1810202, CMMI-1737818, and CMMI-1763070, and by the Air Force Office of Scientific Research under the award FA9550-17-1-0166.}%
\thanks{W. Miao is with the Department of Electrical and Systems Engineering, Washington University, St. Louis, MO 63130, USA
        {\tt\small weimiao@wustl.edu}}%
\thanks{J.-S. Li is with the Department of Electrical and Systems Engineering, Washington University,
        St. Louis, MO 63130, USA
        {\tt\small jsli@wustl.edu}}
}

\maketitle

\begin{abstract}
    In this paper, we tackle the long-standing challenges of ensemble control analysis and design using a convex-geometric approach in a Hilbert space setting. Specifically, we formulate the control of linear ensemble systems as a convex feasibility problem in a Hilbert space, which can be solved by iterative weighted projections. Such a non-trivial geometric interpretation not only enables a systematic design principle for constructing feasible, optimal, and constrained ensemble control signals, but also makes it possible for numerical examination of ensemble reachability and controllability. Furthermore, we incorporate this geometric approach into an iterative framework and illustrate its capability to derive feasible controls for steering bilinear ensemble systems. We conduct various numerical experiments on the control of linear and bilinear ensembles to validate the theoretical developments and demonstrate the applicability of the proposed convex-geometric approach.
\end{abstract}


%
\IEEEpeerreviewmaketitle

\section{Introduction}

Controlling the collective behavior of a population of structurally identical systems, called ensemble control, has received increasing attention in the past decade in the fields of mathematical and applied control. The research streams are driven by the need of investigating deep and unexplored fundamental properties and new control design principles of ensemble control systems, as well as by the relevance of such systems to broad and emerging applications in diverse scientific domains. 

The unique characteristic of the ensemble control problem is in its underactuated nature. Namely, the control and observation of such systems can only be made at the population level, through broadcasting a single input signal to the entire ensemble \cite{glaser1998unitary,li2006control} or receiving aggregated (snapshot-type of) measurements \cite{zeng2014tac}. This nonstandard scenario arises mainly because the number of systems in the ensemble can be exceedingly large so that control via state feedback of each system is infeasible. Notable examples include exciting an ensemble of nuclear spins on the order of Avogadro's number in nuclear magnetic resonance (NMR) spectroscopy and imaging \cite{glaser1998unitary,li2011optimal}, manipulating a group of robots under model perturbation \cite{becker2012approximate}, creating synchronization patterns in a network of coupled oscillators \cite{rosenblum2004controlling,Li_NatureComm16}, and spiking population of neurons to alleviate brain disorders such as Parkinson's disease \cite{brown2004phase,li2013control,kafashan2015optimal}. 

Extensive studies have been conducted on analyzing fundamental properties of ensemble control systems, such as ensemble controllability and ensemble observability \cite{li2009ensemble,zeng2016moment,zeng2017sampled,chen2017distinguished,chen2019structure,schonlein2016controllability,helmke2014uniform}. The developed theoretical methods made good use of algebraic structures of the studied ensemble systems to quantify their reachable sets or observable space defined by parameter-dependent vector fields so as to understand controllability or observability. These methods, however, are in general not suitable for direct implementation as control design principles, so that, independent of control-theoretic analysis, customized computational algorithms are often the ultimate solutions to tackling the challenging ensemble control design problems. In particular, a wide range of numerical methods based on the principle of ensemble control and optimal control theory, such as pseudospectral methods \cite{li2011optimal,ruths2012optimal,phelps2014consistent}, operator-theoretic methods \cite{Li_IFAC17,zeng2018computation}, sample average approximation \cite{phelps2016optimal}, and polynomial approximations \cite{tie2017explicit}, have been devised. Instead of giving an explicit form of a feasible or an optimal ensemble control signal, most of these methods discretize the parameter and the control space, and then solve the resulting non-convex optimization problem defined in a high-dimensional Euclidean space, which may suffer from the curse of dimension and the issue of trapping into local optimality. Besides, specialized techniques dedicated to particular classes of ensemble control problems, such as singular value decomposition (SVD) based algorithm for minimum-energy control of linear ensembles \cite{zlotnik2012synthesis} and iterative methods for quadratic optimal control of bilinear ensembles \cite{Li_SICON17,Li_Automatica18}, do not require solving large-scale optimization problems but have a limited scope, e.g., not able to incorporate constraints on the state and control functions.

In this paper, we develop a unified convex-geometric approach to analyze fundamental properties and synthesize feasible and optimal ensemble controls for linear ensemble systems. Our main idea is to cast the problem of ensemble control design as a ``convex feasibility problem'' in a Hilbert space by leveraging the linearity and convexity inherited in such systems. This is equivalent to finding a feasible point in the intersection of a collection of convex sets defined by the admissible control sets of the individual systems in the ensemble. This nontrivial interpretation of the ensemble control design enables the use of iterative projections in Hilbert spaces to construct feasible and optimal ensemble controls and their convergence properties for numerical evaluation of ensemble reachability. We further extend this design strategy to incorporate constraints on the control inputs, e.g., power or energy limitations, and to find feasible controls for bilinear ensemble systems.

The paper is organized as follows. In Section \ref{sec:feasible control}, we introduce the convex-geometric interpretation of ensemble control and illustrate the use of this interpretation to develop a systematic approach to controllability analysis and feasible control design for linear ensemble systems based on the ideas of iterative weighted projections in Hilbert spaces. In Section \ref{sec: optimal control}, we include an optimization formulation into this framework for the design of minimum-energy ensemble controls. In Section \ref{sec: ensemble control with constraints}, we tailor the proposed convex-geometric approach to accommodating constraints on control inputs and design feasible and optimal ensemble controls with bounds on the total available power or energy. In Section \ref{sec: bilinear}, we extend the convex-geometric approach to design feasible controls of bilinear ensemble systems. In Section \ref{sec: Computation of iterative projection}, we address important aspects of numerical computations for iterative weighted projections in Hilbert spaces, which are important to the synthesis of accurate ensemble controls. Finally, in Section \ref{sec: numerical experiments}, we illustrate the applicability of the proposed convex-geometric approach by a sequence of numerical experiments, including the design of ensemble controls for pattern formation in a linear ensemble and broadband quantum pulses. We also use the pattern formation example to demonstrate the numerical verification of ensemble reachability.

$\frac{}{}         x^{}\dots $

\section{A Convex-Geometric Interpretation of Ensemble Control}
\label{sec:feasible control}
In this section, we present the idea of casting the ensemble control design and controllability analysis to a `convex feasibility problem'. Leveraging on this novel interpretation, we develop methods based on the techniques of iterative weighted projections to systematically construct feasible and, further, optimal controls for linear ensemble systems. In addition, the developed methods offer a rigorous numerical evaluation for reachability between ensemble states of interest, which has not been explored in the literature.

\subsection{The Feasibility Problem and Ensemble Control Design}
\label{sec:linear}
Consider the time-varying linear ensemble system defined in a Hilbert space,
\begin{equation}
  \label{eq:linear}
  \frac{d}{dt}X(t,\b)=A(t,\b)X(t,\b)+B(t,\b)u(t),
\end{equation}
indexed by the parameter $\b$ taking values on a compact set $K\subset\mathbb{R}$, where $X(t,\cdot)\in L^2(K,\mathbb{R}^n)$ is the state, an $n$-tuple of $L^2$-functions over $K$ for each $t\in [0,T]$ with $T\in (0,\infty)$; $u\in L^2([0,T],\mathbb{R}^m)$ is the control, an $L^2$-function; $A\in L^\infty(D,\mathbb{R}^{n\times n})$ and $B\in L^2(D,\mathbb{R}^{n\times m})$ are matrices whose elements are real-valued $L^\infty$- and $L^2$-functions, respectively, defined on the compact set $D=[0,T]\times K$. A typical goal for the control of such an ensemble system is to design a `broadcast' control signal $u$ that steers the entire ensemble from an initial state $X_0(\beta)$ to, or to be within a desired neighborhood of, a target state $X_F(\b)$ at a finite time $T$. This is related to the properties of ensemble reachability and ensemble controllability formally defined as follows.

\begin{df}[Ensemble Reachability and Controllability]
  \label{def:controllability}
  Consider an ensemble of systems defined on a manifold $M$ parameterized by a parameter $\b$ taking values on a space $K$, given by 
  \begin{equation}
    \label{eq:ensemble}
    \frac{d}{dt}X(t,\b)=F(t,\b,X(t,\b),u(t)),
  \end{equation}
  where $X(t,\cdot)\in\mathcal{F}(K)$ is the state and $\mathcal{F}(K)$ is a space of $M$-valued functions defined on $K$. A target state $X_F(\beta)\in\mathcal{F}(K)$ is said to be \textit{ensemble reachable} from an initial state $X(0, \beta)\in\mathcal{F}(K)$ if for any $\e>0$, there exists a piecewise constant control signal $u:[0,T]\to\mathbb{R}^m$ that steers the system into an $\e$-neighborhood of a desired target state $x_F\in\mathcal{F}(K)$ at a finite time $T>0$, i.e., $\rho(x(T,\cdot),X_F(\cdot))<\e$, where $\rho:\mathcal{F}(K)\times \mathcal{F}(K)\rightarrow\mathbb{R}$ is a metric on $\mathcal{F}(K)$. Furthermore, if any $X_F(\beta)\in \mathcal{F}(K)$ is ensemble reachable from arbitrary $X(0, \beta)\in\mathcal{F}(K)$, then the system is said to be \textit{ensemble controllable} on $\mathcal{F}(K)$.
\end{df}


Due to the nonstandard, under-actuated nature of ensemble systems, it is opaque to realize how to assemble the right toolkit from classical systems theory for ensemble control-theoretic analysis and control design. It is to our surprise that such a challenging task becomes transparent from a delicate convex-geometric perspective. To fix ideas, let's now consider a finite sample of sub-systems, $X(t,\b_i)$, $i=1,\ldots,N$, from the ensemble in \eqref{eq:linear} with the parameter $\b_i$ taking values in $K$. In this way, each sampled sub-system is a finite-dimensional time-varying linear system in $\mathbb{R}^n$, following the dynamics   
\begin{equation}
  \label{eq:linear_sample}
  \frac{d}{dt}X(t,\b_i)=A(t,\b_i)X(t,\b_i)+B(t,\b_i)u(t),
\end{equation}
where $X(t,\b_i)\in\mathbb{R}^n$ for each $\b_i$ and for all $t\in [0,T]$. From linear systems theory \cite{brockett2015finite}, the control input $u(t)$ driving the system in \eqref{eq:linear_sample}, with a given $\b_i$, from an initial state $X(0,\b_i) = X_0(\b_i)\in\mathbb{R}^n$ to a target state $X_F(\b_i)\in\mathbb{R}^n$ at time $T$ satisfies the integral equation  
\begin{equation*}
  L_i u = \xi_i,
\end{equation*}
where $L_i: L^2([0, T], \mathbb{R}^m) \to \mathbb{R}^n$ is defined by
\begin{equation}
  \label{eq:Li}
  L_i(u) = \int_0^T \Phi(T,\s,\b_i)B(\s,\b_i)u(\s)d\s,
\end{equation} 
$\xi_i \in \mathbb{R}^n$ is formed by the initial and target states,
\begin{equation}\label{eq:xi_i}
  \xi_{i}=X_F(\b_i)-\Phi(T,0,\b_i)X_0(\b_i),
\end{equation}
and $\Phi$ is the transition matrix associated with the system in \eqref{eq:linear_sample}.

We first observe that the linearity of the operator $L_i$ in \eqref{eq:Li}, with respect to the control $u$, gives convexity of the admissible control set.

\begin{lem}
  \label{lem:convex}
  The admissible control set of the system in \eqref{eq:linear_sample} associated with a given pair of initial and target states $(X_0(\b_i),X_F(\b_i))$, defined by 
  \begin{equation}
    \label{eq:Ci}
    C_i = \big\{u\in L^2([0, T], \mathbb{R}^m) \:|\: L_i u = \xi_i \big\},
  \end{equation}
  is a convex and closed affine subspace, where $L_i$ and $\xi_i$ are defined as in \eqref{eq:Li} and in \eqref{eq:xi_i}, respectively.

  \begin{proof}
    For any two controls $u_1,u_2\in C_i$ and any constant $\lambda\in[0,1]$, it holds that
    \begin{align*}
      L_i(\lambda u_1 + (1-\lambda)u_2) &= \lambda L_iu_1 + (1-\lambda)L_iu_2 \\
      &= \lambda\xi_i + (1-\lambda)\xi_i = \xi_i,
    \end{align*}
    and hence $C_i$ is convex. Because $L_i$ is continuous, $C_i$ is closed since it is the inverse image of $\{\xi_i\}$, which is a closed set in $\mathbb{R}^n$. In addition, since $L_i(u_1-u_2) = 0$ holds for any $u_1, u_2\in C_i$, $C_i$ is an affine subspace by linearity of $L_i$. 
  \end{proof}
\end{lem}

This result, though straightforward to show, is very inspiring for a new interpretation of ensemble control design. That is, if we consider two systems in \eqref{eq:linear_sample} characterized by $\b_i$ and $\b_j$ for $i\neq j$, then a common control law that simultaneously steers the two systems to achieve the respective desired transfers must lie in the intersection of the convex admissible control sets $C_i$ and $C_j$, which is also a convex set. In this case, we have $u\in C_i\cap C_j$ such that $L_i u = \xi_i$ and $L_j u=\xi_j$. The same logic applies to an arbitrary number of systems, as illustrated in Figure \ref{fig: geo1}. Therefore, the design of a broadcast ensemble control input is equivalent to a `convex feasibility problem' over a Hilbert space, namely, a problem of finding a point in the intersection of convex sets. This can be formulated as an optimization problem of the form,
\begin{equation}
  \label{equ: convex feasibility problem}
    \begin{array}{cc}
    \text{min} & 0,\\
    \text{s.t.} & u \in \displaystyle\bigcap_{i=1}^N C_i,
    \end{array}
\end{equation}
where $C_i$ are defined in \eqref{eq:Ci} for $i=1,\ldots,N$.

\begin{figure}[h]
  \centering
  \adjustbox{minipage=1em,valign=t}{\subcaption{}\label{fig: geo1}}%
  \begin{subfigure}[b]{0.4\linewidth}
    \centering\includegraphics[width=0.6\linewidth]{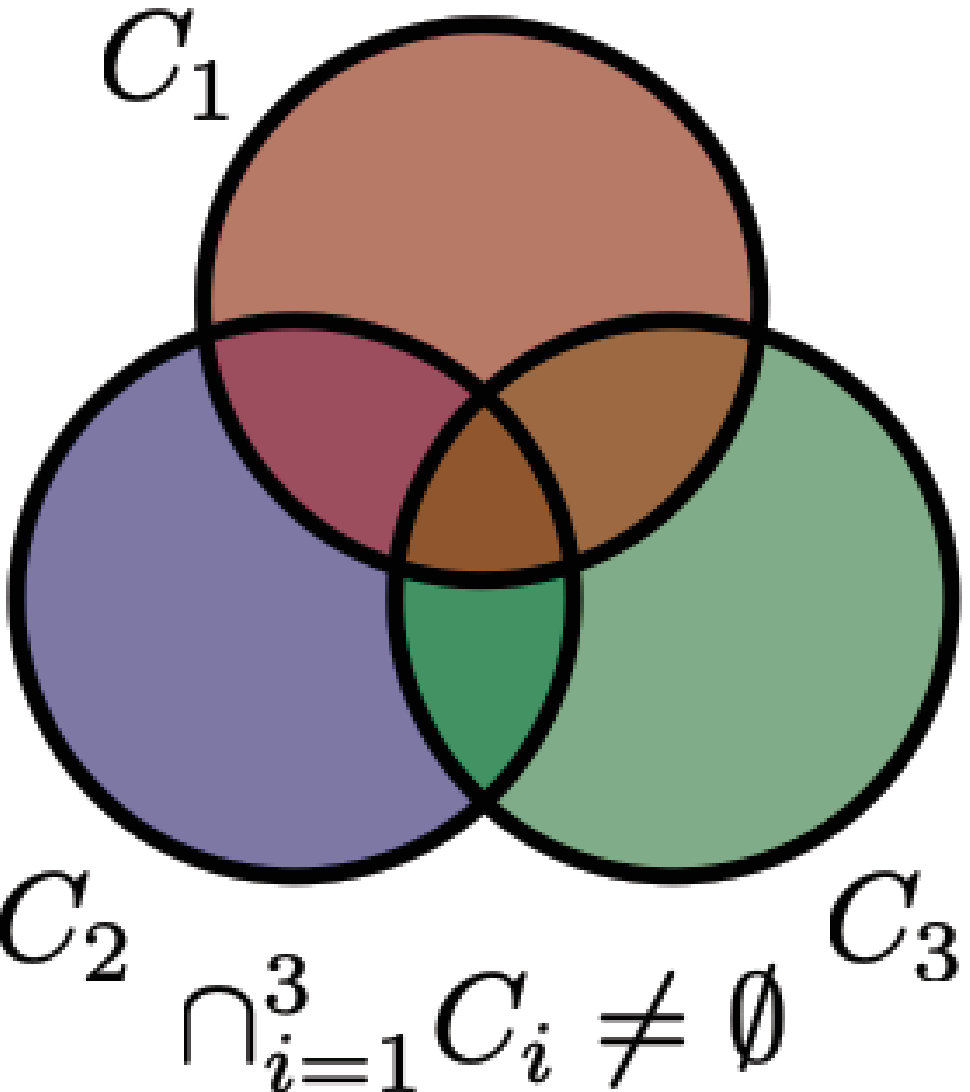}
  \end{subfigure}
  \hspace{5pt}
  \adjustbox{minipage=1em,valign=t}{\subcaption{}\label{fig: geo2}}%
  \begin{subfigure}[b]{0.4\linewidth}
    \centering\includegraphics[width=0.6\linewidth]{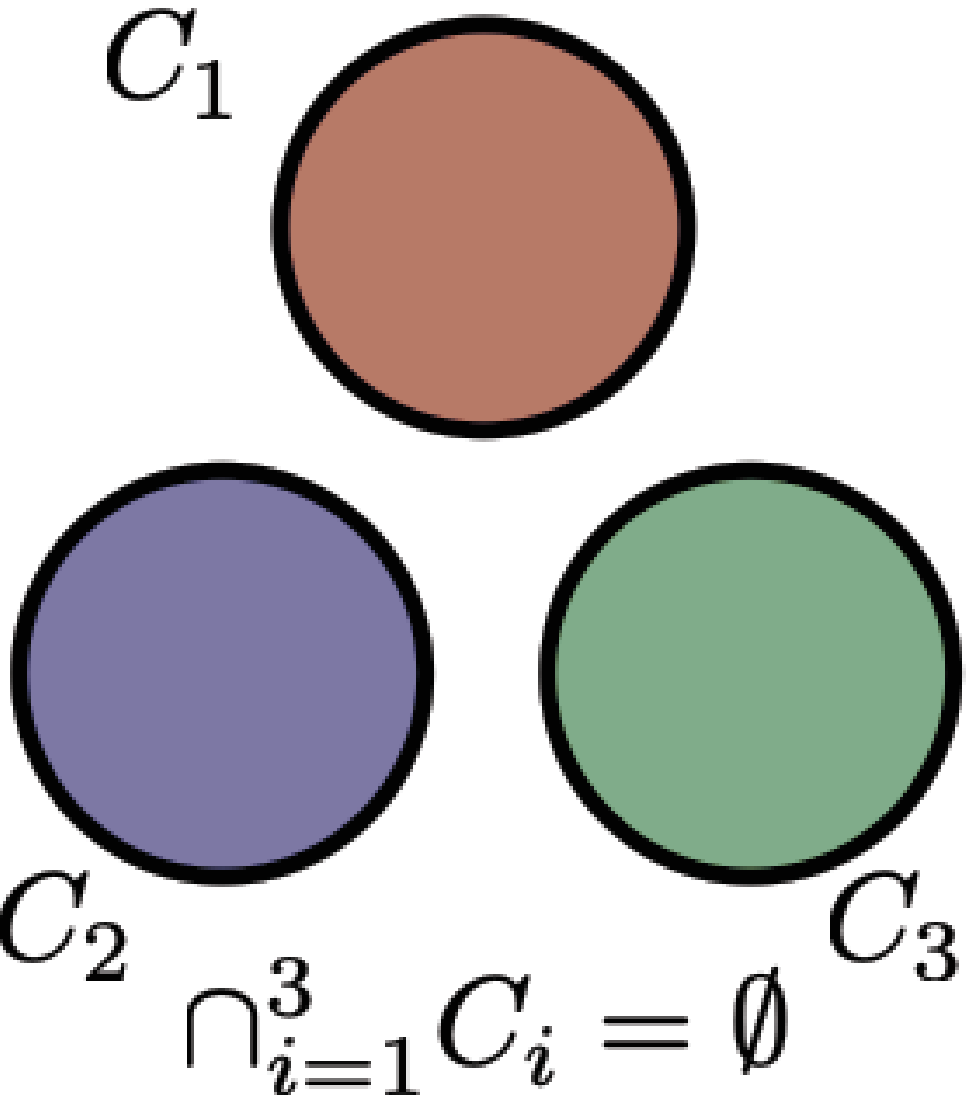} 
  \end{subfigure}
  \caption{Illustration on formulating ensemble control problem as a feasibility problem. $C_i, i = 1,2,3$ denotes the set of control law that steers the sub-system $\beta_i$ to the target state. (a) If the intersection of $3$ admissible control sets is non-empty, then $\cap_{i=1}^3 C_i$ is the collection of feasible ensemble control laws. (b) If the intersection of $3$ admissible control sets is empty, then the system is not ensemble controllable since there exists no control law that steers all sub-systems to target state simultaneously.}
\end{figure}

\subsection{Finding Feasible Points by Iterative Weighted Projections}
\label{sec:iterative_projection}
Because all of the admissible control sets, $C_i$, $i=1,\ldots,N$, are closed and convex, and so is their intersection, a systematic and powerful approach to solve the feasibility problem in \eqref{equ: convex feasibility problem} is to utilize the techniques of iterative weighted projections, as presented in Theorem \ref{thm : weighted projection algorithm}.

\begin{thm}[Iterative weighted projections]
    \label{thm : weighted projection algorithm}
    Let $C_1, \ldots , C_N$ be a collection of closed and convex subsets in a Hilbert space $\mathcal{U}$ and let $P_{C_i}$ be the projection operator onto $C_i$ for $i=1,\ldots,N$. Consider the sequence $\{ u^{(k)} \}$ generated by the convex combination of projections,
    \begin{equation}
        \label{equ: weighted projection}
        u^{(k+1)} = \sum_{i=1}^N \lambda_iP_{C_i} u^{(k)}, \quad k=0,1,2,\ldots,
    \end{equation}
    with an initial point $u^{(0)}\in \mathcal{U}$, where $\lambda_1, \ldots , \lambda_N \in (0,1)$ and $\sum_{i=1}^N \lambda_i = 1$. If $\bigcap_{i=1}^N C_i\neq \emptyset$, then $\{u^{(k)}\}$ converges to a point $u^*\in \bigcap_{i=1}^N C_i$ weakly. Specifically, if $C_i$'s are closed affine subspaces of $\mathcal{U}$, then $\{u^{(k)}\}$ converges in norm.
\end{thm}

\begin{proof}
    The proof can be facilitated by considering the product space $\mathit{\Omega}=\mathcal{U} \times \cdots \times \mathcal{U}$ constituted by a Cartesian product of $N$ copies of $\mathcal{U}$ equipped with the inner product $\langle \cdot , \cdot  \rangle_{\mathit{\Omega}} : \mathit{\Omega}\times \mathit{\Omega} \to \mathbb{R}$ defined by $\langle U,V \rangle_{\mathit{\Omega}} = \sum_{i=1}^N \lambda_i \langle u_i, v_i \rangle_\mathcal{U}$, where $U = (u_1, \ldots , u_N), V = (v_1, \ldots , v_N)$ and $\langle \cdot , \cdot  \rangle_\mathcal{U}$ is the inner product in $\mathcal{U}$. Now, let's construct two closed and convex sets $\mathcal{C}$ and $\mathcal{D}$ in $\mathit{\Omega}$, defined by
    \begin{align}
      \mathcal{C} &= C_1 \times \cdots \times C_N\label{equ: def C}, \\
      \mathcal{D} &= \{U\in \mathit{\Omega} \:|\: u_1= \cdots =u_N \}\label{equ: def D}.
    \end{align}
    Let us associate each $u^{(k)}\in \mathcal{U}$ with a unique element $U^{(k)} \in \mathcal{D}$ defined by $U^{(k)} = (u^{(k)}, \ldots , u^{(k)})$, such that if $\{U^{(k)}\}$ converges to $U^*:=(u^*, \ldots , u^*)\in \mathcal{D}$, then $\{u^{(k)}\}$ converges to $u^*\in \mathcal{U}$.


    For the sequence $\{u^{(k)}\}$ generated by in \eqref{equ: weighted projection}, we can prove that the associated sequence $\{U^{(k)}\}$ satisfies $U^{(k+1)} = P_{\mathcal{D}} P_{\mathcal{C}}U^{(k)}$ (see Appendix \ref{appendix: pf weighted projection}). So the sequence $$\{U^{(k)} = \underbrace{P_{\mathcal{D}}P_{\mathcal{C}} \cdots P_{\mathcal{D}}P_{\mathcal{C}}}_{k\text{ times}}U^{(0)}\}$$ is an alternating projection onto $\mathcal{C}$ and $\mathcal{D}$. By the von Neuman alternating projection algorithm in Hilbert space \cite{von1950functional}, if $\mathcal{C}\cap \mathcal{D} \neq \emptyset$, then $U^{(k)}$ converges to $U^* \in \mathcal{C}\cap \mathcal{D}$ weakly \cite{bauschke1996projection}. Equivalently if $\bigcap_{i=1}^N C_i\neq \emptyset$, then $u^{(k)}$ converges to $u^*\in \bigcap_{i=1}^N C_i$ weakly. Furthermore, if $\mathcal{C}$ and $\mathcal{D}$ are closed affine subspaces, then $U^{(k)}\to U^*$ in norm when $\mathcal{C}\cap \mathcal{D}\neq \emptyset$, which is a direct application of the results in \cite{von1950functional}. This implies that $u^{(k)}\to u^*$ in norm when $\bigcap_{i=1}^N C_i\neq \emptyset$, if $C_i$'s are closed affine subspaces.
\end{proof}

Figures \ref{fig: geo_weighted1} and \ref{fig: geo_weighted2} provide a schematic illustration of using iterative weighted projections in \eqref{equ: weighted projection} to find a feasible point in the intersection of convex sets, or, equivalently, to solve the convex feasibility problem formulated in \eqref{equ: convex feasibility problem}. Figure \ref{fig: geo_weighted1} depicts the case in which an equally weighted iterative scheme is applied, from an arbitrary initial point $u^{(0)}$, to find a point in the intersection of two intersected convex sets. The process is to project $u^{(0)}$ alternatively 
onto $C_1$ and $C_2$, denoted $P_{C_1}u^{(0)}$ and $P_{C_2}u^{(0)}$, respectively, and then obtain $u^{(1)} = \frac{1}{2}(P_{C_1}u^{(0)} + P_{C_2}u^{(0)})$. Continuing this process generates a sequence of points $\{u^{(k)}\}$, and the iterations finally converge to a point, say $u^{(k)}\to u^*\in C_1\cap C_2$ as $k\to\infty$. On the other hand, if $C_1\cap C_2=\emptyset$, the procedure may still be convergent, but not to a point of interest, i.e., $u^*\notin C_1$ and $u^*\notin C_2$, as displayed in Figure \ref{fig: geo_weighted2}.

\begin{figure}[h]
  \centering
  \adjustbox{minipage=1em,valign=t}{\subcaption{}\label{fig: geo_weighted1}}%
  \begin{subfigure}[b]{0.45\linewidth}
    \centering\includegraphics[width=0.6\linewidth]{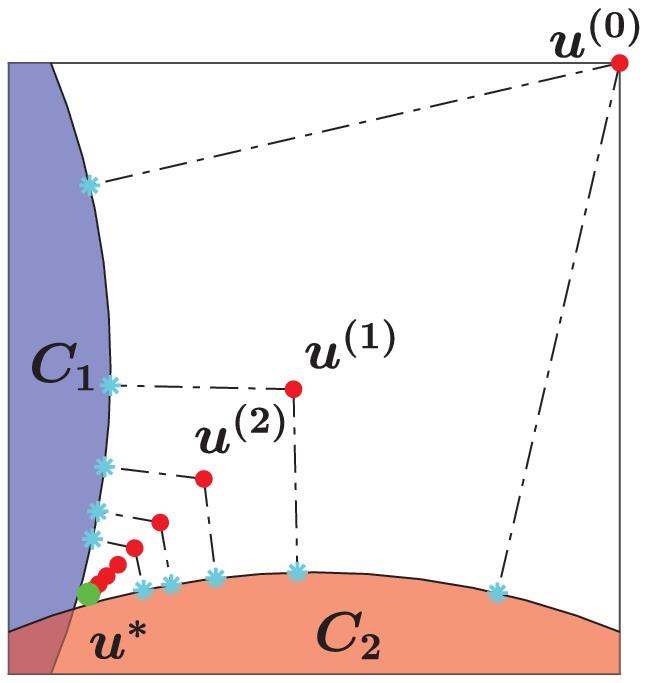}
  \end{subfigure}
  \hspace{-8pt}
  \adjustbox{minipage=1em,valign=t}{\subcaption{}\label{fig: geo_weighted2}}%
  \begin{subfigure}[b]{0.45\linewidth}
    \centering\includegraphics[width=0.6\linewidth]{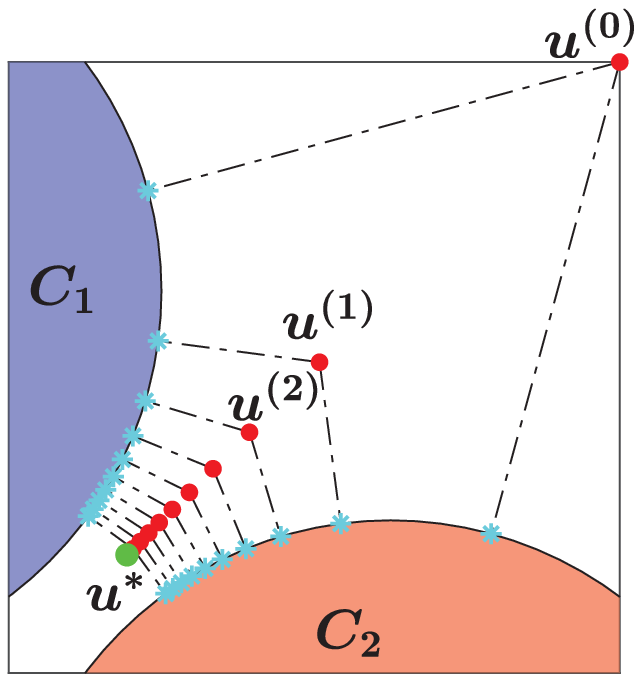} 
  \end{subfigure}
  \caption{Illustration of iterative weighted projections. $C_i, i = 1,2$ denotes the set of control law that steers the sub-system $\beta_i$ to the target state. $u^{(k+1)} = \frac{1}{2}(P_{C_1}u^{(k)} + P_{C_2}u^{(k)})$ for $k = 0,1,2,\ldots$ (a) If $C_1\cap C_2 \neq \emptyset$, then the weighted projection converges to a point in $C_1\cap C_2$. (b) If $C_1\cap C_2 = \emptyset$, then the weighted projection may still converge, however to a point outside of $C_1$ and $C_2$.}
\end{figure}

\subsection{Computing Ensemble Control Laws and Examining Ensemble Controllability}
\label{sec:ec_law}
Theorem \ref{thm : weighted projection algorithm} provides a systematic approach and a powerful means to compute a feasible point in the intersection of finitely many closed and convex sets in a Hilbert space by iterative weighted projections. This feasible point corresponds to a feasible ensemble control law for steering the ensemble system in \eqref{eq:linear_sample} from $X(0,\b_i)$ to $X_F(\b_i)$, $i = 1, \ldots , N$. Most importantly, a distinct feature of the iterative weighted projection algorithm is its capability to verify the existence of a nonempty intersection among the given convex sets through computing $\{u^{(k)}\}$ using \eqref{equ: weighted projection}. 

This validation is of particular importance in the context of ensemble control since it informs whether the ensemble states are reachable and, further, the ensemble is controllable. In particular, for a linear ensemble, because the admissible control set defined in \eqref{eq:Ci} is an affine subspace of $\mathcal{U}$, by the contraposition of Theorem \ref{thm : weighted projection algorithm}, if $\{u^{(k)}\}$ does not converge to a point in $\mathcal{U}$ in norm (either $\|u^{(k)}\|_{\mathcal{U}}\to \infty $ or $u^{(k)}$ oscillates because of weak convergence), then it must hold that $\bigcap_{i=1}^N C_i = \emptyset$, and hence the system in \eqref{eq:linear_sample} is not ensemble controllable because there exists no common control $u$ that will simultaneously steer the entire ensemble to $X_F(\b_i)$.

On the other hand, if $\{u^{(k)}\}$ converges in norm with $u^{(k)}\to u^*$, then there are two possible cases as shown in Figures \ref{fig: geo_weighted1} and \ref{fig: geo_weighted2}. To distinguish them, one can simply apply the convergent control law $u^*$ to the linear ensemble in \eqref{eq:linear_sample}. If $u^*$ steers the ensemble to the desired target state $X_F(\beta)$, then we have the case $\bigcap_{i=1}^N C_i\neq \emptyset$ (Figure \ref{fig: geo_weighted1}). If $u^*$ fails to steer the system to $X_F(\beta)$, then it must hold that $\bigcap_{i=1}^N C_i = \emptyset$ (Figure \ref{fig: geo_weighted2}) by the contraposition of Theorem \ref{thm : weighted projection algorithm}.
As a result, this convergence property renders a rigorous and tractable numerical approach to examine the reachability of an ensemble system and to systematically design an ensemble control law, as described in the following theorem.

\begin{thm} 
  \label{thm: feasible control and controllability for linear ensemble}
  Consider the linear ensemble system in \eqref{eq:linear_sample}. Let $\{u^{(k)}\}$ be a control sequence generated according to the iterative weighted projections in \eqref{equ: weighted projection}, given by the explicit form
  \begin{align}
      \label{equ: update rule for linear ensemble}
      u^{(k+1)} &=  \big(I_d - \sum_{i=1}^N \lambda_i L_i^*(L_i L_i^*)^{-1}L_i\big)u^{(k)} + \sum_{i=1}^N \lambda_i L_i^*(L_i L_i^*)^{-1}\xi_i,
  \end{align}
  with an arbitrary initialization $u^{(0)} \in L^2([0, T], \mathbb{R}^m)$, where $I_d$ is the identity operator in $L^2([0, T]; \mathbb{R}^m)$, $L_i$ and $\xi_i$ are defined in \eqref{eq:Li} and \eqref{eq:xi_i}, respectively; $L_i^*$ denotes the adjoint operator of $L_i$; and $\lambda_i\in (0, 1)$ satisfies $\sum_{i=1}^N \lambda_i = 1$. Then, $X_F(\b_i)$ is ensemble reachable from $X_0(\b_i)$ if and only if $\{u^{(k)}\}$ converges in norm to $u^*\in L^2([0, T], \mathbb{R}^m)$ and $u^*$ satisfies $L_iu^*=\xi_i$ for $i=1,\ldots,N$.
\end{thm}

\begin{proof}
    The proof directly follows Lemma \ref{lem:convex}, Theorem \ref{thm : weighted projection algorithm}, and the analysis above this theorem. 

    What remains to show here is to derive the explicit expression of the projection operator $P_{C_i}$ in \eqref{equ: weighted projection}. 

    We observe that finding the projection of a vector $u\in L^2([0,T],\mathbb{R}^m)$ onto a closed subspace $C_i$, denoted $P_{C_i}u$, is equivalent to solving the least-squares problem,
    \begin{equation}
        \label{equ: Projection operator formulation I}
        \begin{array}{ll}
            \underset{v\in L^2([0,T],\mathbb{R}^m)} {\text{min}} & \|u - v\|_2,\\
            \qquad\ {\rm s.t.} & L_i v = \xi_i.
        \end{array}
    \end{equation}
     For the ensemble control problem, we are interested in the case when the system indexed by each $\beta_i$ is controllable. Hence we have $\mathcal{R}(L_i) = \mathbb{R}^n$. Then from the knowledge of functional analysis, we have $\mathcal{R}(L_iL_i^*) = \mathcal{R}(L_i) =  \mathbb{R}^n$, and the solution of \eqref{equ: Projection operator formulation I} can be written as
    \begin{align}
        P_{C_i} u = (I_d - L_i^*(L_i L_i^*)^{-1}L_i)u + L_i^*(L_i L_i^*)^{-1}\xi_i. \label{equ: projection P_i}
    \end{align}
    Substituting in \eqref{equ: projection P_i} into in \eqref{equ: weighted projection} yields the update rule in \eqref{equ: update rule for linear ensemble}. 
\end{proof}

\section{The Minimum-energy Control for Linear Ensemble Systems}
\label{sec: optimal control}
In Section \ref{sec:feasible control}, we have formulated the ensemble control design as a convex feasibility problem, through which we were able to calculate a feasible ensemble control law. In this section, we take a step further and extend the developed method to find optimal ensemble controls for linear ensemble systems. In particular, we study the minimum-energy control for the ensemble system in \eqref{eq:linear_sample}.

\subsection{The Optimization Formulation of Minimum-Energy Ensemble Control}
\label{subsec: orthonormal projection and Minimum-energy Ensemble Control}
Adopting the same idea of formulating the ensemble control design as a feasibility problem, the minimum-energy control of the ensemble system in \eqref{eq:linear_sample} can be cast as an optimization problem, given by
\begin{equation}\label{eq:minimum-energy control problem}
    \begin{array}{cc}
    \text{min} & \|u\|_2^2,\\
    \text{s.t.} & u \in \displaystyle\bigcap_{i=1}^N C_i,
    \end{array}
\end{equation} 
where $\|u\|_2^2=\int_0^T u'u\, dt$ and $C_i$ are defined as in \eqref{eq:Ci}.

An intriguing fact in this optimization is that the objective function represents the distance between $u$ and the origin, i.e., the zero function in $L^2([0, T], \mathbb{R}^m)$. Therefore, minimizing the energy of $u$ is nothing but finding the point in $\bigcap_{i=1}^N C_i$ that is closest to the origin in $L^2([0, T], \mathbb{R}^m)$, which is the `orthogonal projection' of the origin onto the set $\bigcap_{i=1}^N C_i$.

\subsection{Orthogonal Projections onto Intersections of Convex Sets}
Let $\mathcal{U}$ be a Hilbert space and $C_i$ are closed and convex sets in $\mathcal{U}$. It is in general difficult to directly compute the projection of a vector $u\in\mathcal{U}$ onto $\bigcap_{i=1}^N C_i$. Fortunately, when the projection onto each $C_i$, i.e., $P_{C_i} u$, is easy to compute, then the projection onto the intersection $\bigcap_{i=1}^N C_i$ can be easily obtained by cyclic projections onto individual sets. Dykstra's algorithm \cite{boyle1986method} is a notable method performing such computations, described in Algorithm \ref{algo:Dykstra's algorithm}. 

\begin{algorithm}  
  \caption{Dykstra's algorithm}  \label{algo:Dykstra's algorithm}
  \begin{algorithmic} 
    \Function{Dykstra}{$C_1, \ldots , C_N, u^{(0)}$}
    \State \textbf{Initialize}: 
    \begin{enumerate}\setlength\itemindent{25pt}
      \item [(0)] Assign $u_0^{(1)} = u^{(0)}$.
      \item [(1)] Project $u_0^{(1)}$ onto $C_1$ to obtain $u_1^{(1)}$. 
      \item [] Compute $I_1^{(1)} = u_1^{(1)} - u_0^{(1)}$.
      \item [] $\vdots$
      \item [(N)] Project $u_{N-1}^{(1)}$ onto $C_N$ to obtain $u_N^{(1)}$. 
      \item [] Compute $I_N^{(1)} = u_N^{(1)} - u_{N-1}^{(1)}$.
    \end{enumerate}
    \For{$k \gets 2 ,3, \dots$}:
    \begin{enumerate}\setlength\itemindent{25pt}
      \item [(0)] Assign $u_0^{(k)} = u_N^{(k-1)}$.
      \item [(1)] Project $u_0^{(k)} - I_1^{(k-1)}$ onto $C_1$ to obtain $u_1^{(k)}$. 
      \item [] Compute $I_1^{(k)} = u_1^{(k)} - (u_0^{(k)} - I_1^{(k-1)})$.
      \item [] $\vdots$
      \item [(N)] Project $u_{N-1}^{(k)} - I_N^{(k-1)}$ onto $C_N$ to obtain $u_N^{(k)}$. 
      \item [] Compute $I_N^{(k)} = u_N^{(k)} - (u_{N-1}^{(k)} - I_N^{(k-1)})$.
    \end{enumerate}
    \EndFor
    \State\Return $\{u_1^{(k)}\}_{k=1}^{\infty }, \ldots ,\{u_{N-1}^{(k)}\}_{k=1}^{\infty }$
    \EndFunction
  \end{algorithmic}  
\end{algorithm}

\begin{lem}\label{lem: Dykstra's algorithm}
  Let $C_1, \ldots , C_N$ be closed and convex in a Hilbert space $\mathcal{U}$ and let $u^*$ be the orthogonal projection of $u^{(0)}$ onto $\bigcap_{i=1}^N C_i$, then the sequence $\{u_i^{(k)}\}$ obtained by Algorithm \ref{algo:Dykstra's algorithm} converges to $u^*$ in norm for any $1\leq i \leq N$. Furthermore, if $C_i$ is affine for all $i=1,\ldots,N$, then the offset $I_i^{(k-1)}$ in Algorithm \ref{algo:Dykstra's algorithm} can be set as $0$ for all $k=2,3,\ldots $, with the same guarantee that $\{u_i^{(k)}\}$ converges to $u^*$ in norm for all $i = 1, \ldots , N$.
\end{lem}
\begin{proof}
  See the proof of Theorem 2 in \cite{boyle1986method}.
\end{proof}

\begin{rem} (\textbf{Dykstra's algorithm v.s. iterative weighted projection algorithm})
  It is worth pointing out that the Dykstra's algorithm involves an offset $I_i^{(k)}$ when projecting $u_i^{(k+1)}$ onto $C_i$ in the next iteration, which eventually results in the point in $\bigcap_{i=1}^N C_i$ that is the closest to $u^{(0)}$; while the weighted projection algorithm, with no offset terms, only returns one feasible 
  point in $\bigcap_{i=1}^N C_i$, not necessarily the one closest to $u^{(0)}$.
\end{rem}

\subsection{Computing Minimum-Energy Control for Linear Ensembles by Iterative Weighted Projections}
\label{subsec: optimal control linear}
From the discussion in Section \ref{subsec: orthonormal projection and Minimum-energy Ensemble Control}, we have interpreted the minimum-energy ensemble control problem as the projection of the zero function, i.e., $0\in L^2([0, T], \mathbb{R}^m)$, onto the intersection of the admissible control sets, formulated as in \eqref{eq:minimum-energy control problem}, which can be computed using Dykstra's algorithm.

As illustrated in the proof of Theorem \ref{thm : weighted projection algorithm}, the procedure of projecting $0$ onto $\bigcap_{i=1}^N C_i$ is equivalent to projecting $0$ onto $\mathcal{C}\cap \mathcal{D}$ using Dykstra's algorithm, where $\mathcal{C}$ and $\mathcal{D}$ are defined in \eqref{equ: def C} and in \eqref{equ: def D}, respectively. Because both $\mathcal{C}$ and $\mathcal{D}$ are affine subspaces, by Lemma \ref{lem: Dykstra's algorithm}, setting the offsets $I_i^{(k-1)} = 0$ for $i=1,2$ and $k=2,3,\ldots$ yields a sequence $\{u^{(k)}\}$ converging to the projection of $0$ on $\mathcal{C}\cap \mathcal{D}$,  which gives rise to the following theorem.

\begin{thm} \emph{(Minimum-energy control for linear ensemble systems)}
  \label{thm: minimum-energy}
  Consider the linear ensemble system in \eqref{eq:linear_sample} with $L_i$, $\xi_i$, and $C_i$ defined as in \eqref{eq:Li}, \eqref{eq:xi_i} and \eqref{eq:Ci}, respectively. If $\bigcap_{i=1}^N C_i \neq \emptyset$, then the control sequence $\{u^{(k)}\}$, with $u^{(0)}=0 \in L^2([0, T], \mathbb{R}^m)$, generated by the iterations,
    \begin{align*}
       u^{(k+1)} = (I_d &- \sum_{i=1}^N \lambda_i L_i^*(L_i L_i^*)^{-1}L_i)u^{(k)}+ \sum_{i=1}^N \lambda_i L_i^*(L_i L_i^*)^{-1}\xi_i,
    \end{align*}
    for $k=0,1,2,\ldots,$ converges to the minimum-energy ensemble control that steers the ensemble from $X_0(\b_i)$ to $X_F(\b_i)$ in norm, where $I_d$ is the identity operator in $L^2([0, T], \mathbb{R}^m)$; $L_i^*$ denotes the adjoint operator of $L_i$; and $\lambda_i\in(0, 1)$ satisfying $\sum_{i=1}^N \lambda_i = 1$.
\end{thm}

\begin{proof}
  We first define the product space $\mathit{\Omega}$ and the augmented sets $\mathcal{C}$ and $\mathcal{D}$ in the same way as in the proof of Theorem \ref{thm : weighted projection algorithm}. Then, we associate each $u\in \bigcap_{i=1}^N C_i$ a unique element $U\in\mathcal{D}$ defined by $U=(u, \ldots ,u)$, so that we have $U\in \mathcal{C}\cap \mathcal{D}$, and $\|U\|^2_{\mathit{\Omega}} = \sum_{i=1}^N \lambda_i \|u\|^2_\mathcal{U} = \|u\|^2_\mathcal{U}$. This reformulates the optimization problem in \eqref{eq:minimum-energy control problem} into a new optimization problem over the product space $\mathit{\Omega}$ to minimize $\|U\|_\Omega^2$ subject to $U\in \mathcal{C} \cap \mathcal{D}$, which solution is provided by the orthogonal projection of $0\in \mathit{\Omega}$ that can be computed by the Dykstra's algorithm. Since $\mathcal{C}$ and $\mathcal{D}$ are both closed affine subspaces of $\mathit{\Omega}$, the Dykstra's algorithm on $\mathcal{C} \cap \mathcal{D}$ with $U^{(0)}=0 \in \mathit{\Omega}$ boils down to the iterative weighted projections on $\mathcal{C}$ and $\mathcal{D}$, which leads to the projection of $u^{(0)}=0 \in \mathcal{U}$ onto $\bigcap_{i=1}^N C_i$. The convergence of iterative projections on $\mathcal{C}$ and $\mathcal{D}$ has already been fully analyzed in Theorem \ref{thm : weighted projection algorithm}. Hence simply changing the initial condition in \eqref{equ: update rule for linear ensemble} to $u^{(0)} = 0$ provides a sequence $\{u^{(k)}\}$ converging to the minimum-energy control law of the ensemble. \qedhere
\end{proof}

\section{Ensemble Control under Constraints on Control Inputs}
\label{sec: ensemble control with constraints}
Despite the success of previous developments in ensemble control design \cite{zlotnik2012synthesis,zeng2017sampled,zeng2016moment,Li_IFAC17}, algorithms accommodating practical limitations on the control input, such as energy and power constraints, remain underdeveloped. In this section, we fill this gap and illustrate the flexibility of the convex-geometric framework established in Sections \ref{sec:feasible control} and \ref{sec: optimal control} for incorporating control constraints, which makes it a unified ensemble control design approach.

Specifically, we denote the set of constraints on the control signal as $G \subset L^2([0, T], \mathbb{R}^m)$. If $G$ is a convex set, then similar to the formulation in \eqref{equ: convex feasibility problem}, a constrained ensemble control design problem can be cast as a feasibility problem of the form,
\begin{equation}\label{equ: convex feasibility problem constrained}
    \begin{array}{cc}
    \text{min} & 0,~~~~~~~~\\
    \text{s.t.} & u \in \big(\bigcap_{i=1}^N C_i\big)\cap G,
    \end{array}
\end{equation}
which can be solved by the weighted projection algorithm presented in Theorem \ref{thm : weighted projection algorithm}.

\begin{cor} \label{cor:ensemble control under constraints}
    Consider the linear ensemble system in \eqref{eq:linear_sample} and denote $G$ as the set of constraints on the control input $u$, i.e., $u\in G$. If $(\bigcap_{i=1}^N C_i) \cap G \neq \emptyset$, then the sequence generated according to the weighted projection, given by
    \begin{align}
      \label{equ: update rule for linear ensemble constrained}
       u^{(k+1)} =  \lambda_0P_G u^{(k)} &+ \sum_{i=1}^N \lambda_i [(I_d - L_i^*(L_iL_i^*)^{-1}L_i)u^{(k)}+ L_i (L_iL_i^*)^{-1}\xi_i],
    \end{align}
    converges weakly to a feasible constrained ensemble control, where $L_i$, $\xi_i$, and $C_i$ are defined as in \eqref{eq:Li}, in \eqref{eq:xi_i}, and in \eqref{eq:Ci}, respectively, and $\lambda_0, \ldots , \lambda_N \in (0, 1)$ and $\sum_{i=0}^N \lambda_i = 1$.
\end{cor}

\begin{proof}
    The proof directly follows the proof of Theorem \ref{thm : weighted projection algorithm} by applying iterative weighted projections on the convex sets, $G, C_1, \ldots, C_N$.
\end{proof}

\begin{rem}[\textbf{Norm convergence for final state of linear ensembles}]
    Since $G$ is not necessarily a closed affine subspace, Corollary \ref{cor:ensemble control under constraints} only guarantees the weak convergence of the control sequence, i.e., $\{u^{(k)}\}\to u^*$ as $k\to\infty$. Nevertheless, the final ensemble state $X^{(k)}(T, \beta_i)$, under the control law $u^{(k)}$, is guaranteed to converge to the target ensemble state $X_F(\beta_i)$ in norm. This results from the definition of weak convergence, namely, if $u^{(k)}$ converges to $u^*$ weakly, then $Lu^{(k)}$ converges to $Lu^*$ in norm for any bounded linear operator $L:H \to \mathbb{R}$. Recall that the final ensemble state of the linear ensemble system in \eqref{eq:linear_sample} is $X^{(k)}(T, \beta_i) = \Phi(T, 0, \beta_i)X(0) + L_iu^{(k)}$, which is determined by the bounded linear operator $L_i:L^2([0,T],\mathbb{R}^m)\to\mathbb{R}^n$. Therefore $X^{(k)}(T, \beta_i)\to X_F(\beta_i)$ in norm as $k\to\infty$
\end{rem}

With the guarantee that iterative weighted projections converge, next, we show how to explicitly compute projections onto the constraint set $G$. In particular, we will present two cases in which the control energy or amplitude is limited.

\begin{lem}
  Let $G_2 := \{u\in L^2([0, T], \mathbb{R}^m)\:|\: \|u\|_{2} \leq M\}$, where $M >0$ is a constant. Then, for any $u\in L^2([0, T], \mathbb{R}^m)$ and $u\not\in G_2$, the orthogonal projection of $u$ onto $G_2$ is given by 
  \begin{equation}\label{equ: claim 2-norm}
        P_{G_2} u = \frac{u}{\|u\|_2}M.
    \end{equation}
\end{lem}
\begin{proof}
  For any $v\in G_2$, by triangle inequality, the norm of $u-v$ is bounded below by
        \begin{align}
            \label{equ: claim pf 1}
            \|u - v\|_2 & \geq \|u - 0\|_2 - \|v - 0\|_2  = \|u\|_2 - \|v\|_2.
        \end{align}
        Since $v\in G_2$, it holds that $\|v\| \leq M$. Hence \eqref{equ: claim pf 1} can be further bounded by
        \begin{equation}\label{equ: claim pf 2}
            \|u - v\| \geq \|u\|_2 - M.
        \end{equation}
        On the other hand, if we directly compute the norm of $u - \frac{u}{\|u\|_2} M$, it can be discovered that
        \begin{equation}\label{equ: claim pf 3}
            \|u - \frac{u}{\|u\|_2} M\|_2 = \|u\|_2(1 - \frac{M}{\|u\|_2}) = \|u\|_2 - M.
        \end{equation}
        Combining \eqref{equ: claim pf 2} and \eqref{equ: claim pf 3} yields
        \begin{align*}
            \|u - v\|_2 \geq \|u - \frac{u}{\|u\|_2} M\|_2, \quad\forall\, v\in G_2.
        \end{align*}
        Hence the projection is given by $P_{G_2} = \frac{u}{\|u\|_2}M.$
\end{proof}

\begin{lem}
  Let $G_\infty := \{u\in L^2([0, T], \mathbb{R}^m)\:|\:\underset{t\in [0, T]}{\text{\rm max}} |u(t)| \leq M\}$, where $M >0$ is a constant. Then, for any $u\in L^2([0, T], \mathbb{R}^m)$ and $u\not\in G_\infty $, the orthogonal projection of $u$ onto $G_\infty$ is given by 
  \begin{equation}\label{equ: claim inf-norm}
        \left(P_{G_{\infty }}u\right)_i(t) = \left\{ \begin{array}{ccl}
        M, & \text{if}& u_i(t) > M\\
        u_i(t), & \text{if}& -M\leq u_i(t)\leq M\\
        -M, & \text{if}& u_i(t) < -M
        \end{array} \right.,
    \end{equation}
    for $i = 1, \ldots , m$, where $(P_{G_{\infty}}u)_i$ and $u_i$ denote the $i^{\text{th}}$ component of $P_{G_{\infty}}u$ and $u$, respectively.
\end{lem}
\begin{proof}
  For any $v\in G_{\infty }$, we observed that
        \begin{align*}
            &\|u - v\|^2_2 = \int_0^T \|u(t) - v(t)\|_2^2dt \nonumber\\
            &= \int_0^T \sum_{i=1}^m |u_i(t) - v_i(t)|^2 dt = \sum_{i=1}^m \int_0^T |u_i(t) - v_i(t)|^2dt \nonumber\\
            &\geq  \sum_{i=1}^m \left( \int_{I_1 }|u_i(t) - v_i(t)|_2^2dt + \int_{I_2} |u_i(t) - M|_2^2 dt + \int_{I_3} |u_i(t) + M|_2^2 dt \right),
        \end{align*}
        where $I_1 := \{t\in [0, T]\:\big\vert\: \|u(t)\| \leq  M\}$, $I_2 := \{t\in [0, T]\:\big\vert\: u(t) > M\}$ and $I_3 := \{t\in [0, T]\:\big\vert\: u(t) < -M\}$. Hence, we conclude $P_{G_{\infty } u}$ as claimed. 
\end{proof}

\section{Feasible controls for Bilinear Ensembles}
\label{sec: bilinear}
In Sections \ref{sec:feasible control}-\ref{sec: ensemble control with constraints}, we have developed a convex-geometric approach to ensemble control analysis and design. In this section, we show that this approach is not restricted to linear ensemble systems and can be adopted within an iterative framework to find feasible controls for bilinear ensemble systems. More specifically, here we present an iterative method that decodes a bilinear ensemble control problem into a sequence of linear ensemble problems, so that each of which can be solved by the proposed convex-geometric projection approach. To fix ideas, we consider the bilinear ensemble system, 
\begin{align}
  \label{eq:bilinear_ensemble}
  \frac{d}{dt} X(t,\b) = A(\b)X(t,\b)+ \Big(\sum_{i=1}^m u_i(t)B_i(\b)\Big)X(t,\b),
\end{align}
where $X(t, \cdot)\in L^2(K, \mathbb{R}^n)$, $A\in L^\infty(K,\mathbb{R}^{n\times n})$, $B_i\in L^2(K,\mathbb{R}^{n\times n})$, and $u_i\in L^2([0,T],\mathbb{R})$ for $i=1,\ldots,m$, and consider a canonical problem of steering this ensemble from an initial state $X_0(\beta)$ to a target state $X_F(\b)$. Obtaining explicit forms of the steering controls $u_i(t)$ and the resulting trajectory $X(t,\b)$ is in general arduous even numerically, with the exception of only some special cases. To resolve this, our approach is based on expressing the bilinear ensemble system as an iteration equation and then solve it in an iterative manner. Specifically, we first observe that \eqref{eq:bilinear_ensemble} can be expressed in the form of a time-varying state-dependent linear ensemble system, 
$$\frac{d}{dt} X(t,\b) = A(\b)X(t,\b)+ \tilde{B}(X(t,\b))u(t),$$
where $\tilde{B}\in L^2(K,\mathbb{R}^{n\times m})$ and $u=(u_1,\ldots,u_m)'$ (see Example \ref{ex: bloch equations} in Section \ref{sec: numerical experiments}). Now, putting this into an iteration equation,
\begin{align}
    \label{eq:iteration}
    \frac{d}{dt} X^{(k+1)}(t,\b) = A(\b)X^{(k+1)}(t,\b)+ \tilde{B}(X^{(k)}(t,\b))u^{(k+1)},
\end{align}
where $k=0,1,2,\ldots$ denotes the iteration, if the state at the $k^{th}$-iteration, $X^{(k)}(t,\b)$, is known, then the system dynamics of the next iteration $X^{(k+1)}(t,\b)$ obeys \eqref{eq:iteration} with determined $A$ and $\tilde{B}$ matrices, which is in the same form as \eqref{eq:linear}. As a result, the steering control $u^{(k+1)}$ (feasible, optimal, or constrained) and the resulting trajectory $X^{(k+1)}$ associated with the specified initial state $X_0(\beta)$ and desired target state $X_F(\beta)$, can be computed using the convex-geometric approach developed in Sections \ref{sec:feasible control}-\ref{sec: ensemble control with constraints} for the linear ensemble in \eqref{eq:iteration}. Following this iterative procedure to compute the control and update the trajectory in an alternating fashion by initializing with a feasible steering control, $u^{(0)}(t)$, and the corresponding ensemble trajectory, $X^{(0)}(t, \beta)$, a control-trajectory sequence, 
\begin{align}
  \label{eq:sequence}
  (u^{(0)},X^{(0)})\to (u^{(1)},X^{(1)})\to\cdots\to (u^{(k)},X^{(k)})\to\cdots,
\end{align}
can be generated.

The analysis of convergence of this iterative procedure, i.e., of the control-trajectory sequence in \eqref{eq:sequence}, can be facilitated by a quadratic optimal control setting, that is, through considering the minimum-energy ensemble control problem,
\begin{align}
  \min & \ \ J=\frac{1}{2}\int_0^{T} (u^{(k)})^T(t)Ru^{(k)}(t)\,dt, \nonumber\\
  \label{eq:oc4}
  {\rm s.t.} & \  \frac{d}{dt}{X^{(k)}(t,\b)}=A(\b)X^{(k)}(t,\b)+\tilde{B}(X^{(k-1)}(t,\b)) u^{(k)}, \nonumber \\ 
  & \quad\ X^{(k)}(0,\b)=X_0(\b),\quad X^{(k)}(T,\b)=X_F(\b), \nonumber
\end{align}
which the regulator $R\in\mathbb{R}^{m\times m}$ is positive definite. This optimal linear ensemble control problem has been treated in our previous work \cite{Li_SICON17} using the proposed iterative method, where we showed that the sequence in \eqref{eq:sequence} is convergent if the iterated linear ensemble system in \eqref{eq:iteration} is controllable at each iteration $k$ and the control regulator $R$ is sufficiently large \cite{Li_SICON17}. We will provide several numerical examples in Section \ref{sec: numerical experiments} to illustrate this `hybrid' computational framework, integrating the convex-geometric approach into the described iterative method, to solve feasible control problems involving a bilinear ensemble system.

\section{The Computation of Iterative Projections}
\label{sec: Computation of iterative projection}
In the previous sections, we have demonstrated that the ensemble control analysis and design in the presence or absence of constraints 
can be addressed by the optimization formulation and iterative weighted projections. 
Although theoretical guarantees on the convergence of weighted projections are well-established, numerical issues for computing iterative projections are often the bottleneck toward finding the convergent solution as for the case of computing the feasible ensemble law in \eqref{equ: update rule for linear ensemble}, which require a thorough investigation and analysis.

The most common way to apply the update rule in \eqref{equ: update rule for linear ensemble} is to discretize the continuous operators $L_i$ and $L_i^*$ by finite time-steps, and then evaluate the function $u^{(k)}(t)$ by finitely many sampled values. However, when $L_iL_i^*$ is ill-conditioned, the iterations 
suffer from slow convergence over $k$. To overcome this, we derive the closed-form solution of $u^{(k)}$ and use it to directly characterize the asymptotic properties of $\{u^{(k)}\}_{k=1}^{\infty }$.

For ease of exposition, we denote 
$$Q = \sum_{i=1}^N \lambda_i L_i^*(L_i L_i^*)^{-1}L_i \quad\text{and}\quad \delta = \sum_{i=1}^N \lambda_i L_i^*(L_i L_i^*)^{-1}\xi_i,$$ 
and rewrite the update rule in \eqref{equ: update rule for linear ensemble} as
\begin{equation}\label{equ: update rule for linear ensemble2}
  u^{(k+1)} = (I_d - Q)u^{(k)} + \delta.
\end{equation}
Then, some analysis can be conducted for the projection term onto $u^{(k)}$, i.e., $(I_d-Q) u^{(k)} $.

\begin{lem}
  There exists a bounded linear operator $Q^{\infty }$ such that $Q^{\infty} = \lim_{k\to \infty}(I_d - Q)^k$.
\end{lem}
\begin{proof}
  It suffices to show that the operator $I_d - Q$ is non-expansive. Let $P_i := I_d - L_i^*(L_iL_i^*)^{-1}L_i$. We observe that $P_i^2 = P_i$ and $P_i^* = P_i$. It follows that for any $u\in L^2([0, T], \mathbb{R}^m)$, we have 
$$\|P_i u\|^2 = \langle P_iu, P_iu \rangle = \langle u, P_i^*P_iu \rangle = \langle u, P_iu \rangle \leq \|P_iu\| \|u\|,$$ 
which implies that $\|P_iu\| \leq \|u\|$ for all $i = 1, \ldots,N$. As a result, 
$I_d-Q$ is non-expansive since 
  \begin{align*}
    \|(I_d - Q)u\| &= \|\sum_{i=1}^N \lambda_i P_i u\| \leq \sum_{i=1}^N \lambda_i \|P_i u\|\leq \sum_{i=1}^N \lambda_i \|u\| = \|u\|.
  \end{align*}
  Consequently, $\|\lim_{k\to \infty }(I_d - Q) u\|$ is bounded for any $u\in L^2([0, T], \mathbb{R}^m)$, which concludes the proof. 
\end{proof}

Now, we may employ the well-defined $Q^{\infty }$ to compute the closed-form solution of $\lim_{k\to \infty } u^{(k)}$. Specifically, by multiplying $Q$ and taking limits on both sides of \eqref{equ: update rule for linear ensemble2}, it yields the closed-form representation of $\lim_{k\to \infty }Qu^{(k)}$, that is,
\begin{align}
  \lim_{k\to \infty }Qu^{(k)} &= \lim_{k\to \infty }Q( (I_d - Q)u^{(k)} + \delta ) \nonumber\\
  & = \lim_{k\to \infty }(I_d - Q) (Qu^{(k-1)} - \delta) + \delta \nonumber\\
  & = \lim_{k\to \infty }(I_d - Q)[(I_d - Q)(Qu^{(k-2)} - \delta) + \delta -\delta] + \delta\nonumber\\
  & = \lim_{k\to \infty }(I_d - Q)^2 (Qu^{(k-2)} - \delta) + \delta \nonumber\\
  & \:\:\vdots\nonumber\\
  & = \lim_{k\to \infty } (I_d - Q)^k (Qu^{(0)} - \delta) + \delta\nonumber\\
  &= Q^{\infty }(Qu^{(0)} - \delta) + \delta. 
  \label{equ: closed form of Qu}
\end{align}

When computing $\lim_{k\to \infty }u^{(k)}\doteq u^*$ numerically, $Q$ can be approximated by a finite-dimensional matrix $W = \left[ Q(v_1), \ldots Q(v_r) \right]$, where $V=\left\{v_1, \ldots , v_r\right\}$ is a 
truncated basis of $L^2([0, T], \mathbb{R}^m)$. Since $W$ is finite-dimensional, it always admits a Moore-Pseudo inverse, denoted $W^\dagger$. Then, the coordinate of $u^*$ under the basis $V$, denoted as $\mu^*\in\mathbb{R}^r$, 
can be computed by
\begin{align}
  \mu^* &= W^{\dagger} (W^{\infty } (W \mu^{(0)} - \delta) + \delta)= W^{\infty }(\mu^{(0)} - W^{\dagger} \delta) + W^{\dagger}\delta, \label{equ: closed form solution of weighted projection}
\end{align}
where 
$\mu^{(0)}\in\mathbb{R}^r$ is the coordinate of the initial control law $u^{(0)}$ and $W^{\infty } := \lim_{k\to \infty}(I_d - W)^{k}$. 

In addition, from the eigen-decomposition of $I_d - W = P^{-1}\Lambda P$, it is clear to observe that $\lim_{k\to \infty }\Lambda^k$ can be computed by simply eliminating the eigenvalues of $\Lambda$ that are not equal to $1$. In this way, $W^{\infty } = P^{-1} \left(\lim_{k\to \infty }  \Lambda^k\right) P$ can be directly evaluated, so is $\mu^*$ in \eqref{equ: closed form solution of weighted projection} without any iterative computations.

\section{Numerical experiments}
\label{sec: numerical experiments}
In this section, we provide various numerical experiments to illustrate the application of the developed convex-geometric approach to compute ensemble controls. We demonstrate the design of the fixed-endpoint and the pattern formation controls of unconstrained linear ensemble systems. We further display examples on the fixed-endpoint control of linear ensemble systems with energy or amplitude constraints on control inputs. Finally we present results for the fixed-endpoint control of bilinear systems. All the numerical experiments are implemented in Matlab R2017b on a single workstation with Xeon Gold 6144 3.5GHz processor and 192GB memory.

\subsection{Linear Ensemble Systems}
\label{sec:linear_ex}
In this section, we consider several examples of the control of linear ensembles with or without constraints on the control inputs and use them to illustrate the applicability of the proposed convex-geometric approach.
\subsubsection{Unconstrained linear ensemble control}
Consider an ensemble of forced harmonic oscillators, with their frequencies $\w_i$'s ranging in the interval $[\w_a,\w_b]$, modeled by
\begin{align}\label{equ: harmonic oscillator}
    \frac{d}{dt} \begin{bmatrix}
        x_1(t, \omega_i) \\ x_2(t, \omega_i)
    \end{bmatrix} = \begin{bmatrix}
        0 & -\omega_i \\ \omega_i & 0
    \end{bmatrix} \begin{bmatrix}
        x_1(t, \omega_i) \\ x_2(t, \omega_i)
    \end{bmatrix} + \begin{bmatrix}
        u_1(t) \\ u_2(t)
    \end{bmatrix}.
\end{align}
The transition matrix of this system is $\Phi(t, \sigma, \omega_i) = e^{A_{i}(t- \sigma)}$, where $A_{i}= \begin{bmatrix}
  0 & -\omega_i \\ \omega_i & 0
\end{bmatrix}$.
Then, $L_i$ and $\xi_i$ in \eqref{eq:Li} and in \eqref{eq:xi_i}, respectively, can be computed for a given control parameter and a pair of initial and target states.

\begin{ex}[Fixed-endpoint control]\label{ex: harmonic oscillator fixed endpoint}
We consider steering an ensemble of $21$ systems in \eqref{equ: harmonic oscillator} from the same initial state $(1,0)'$ to the same target state $(0,1)'$ at time $T=1$, where $\w_i$'s are uniformly sampled in $[-1,1]$. 

Figure \ref{fig: fixed end-point control} shows a feasible ensemble control signal $(u_1^*(t),u_2^*(t))$ designed by the method of iterative weighted projections, which achieves the desired transfer. In this case, the initial control functions are $u_1^{(0)} (t) \equiv u_2^{(0)}(t) \equiv 1$ and $(u_1^*(t),u_2^*(t))$ are obtained based on the update rule in \eqref{equ: update rule for linear ensemble} after $1\times 10^5$ iterations. Figure \ref{fig: fixed end-point error} presents the terminal errors with respect to different numbers of iterations, in which we show that more iterations result in smaller terminal errors. 

\begin{figure}[ht]
  \centering
  \adjustbox{minipage=1em,valign=t}{\subcaption{}\label{fig: fixed end-point control}}%
  \begin{subfigure}[b]{0.45\linewidth}
    \centering\includegraphics[width=0.6\linewidth]{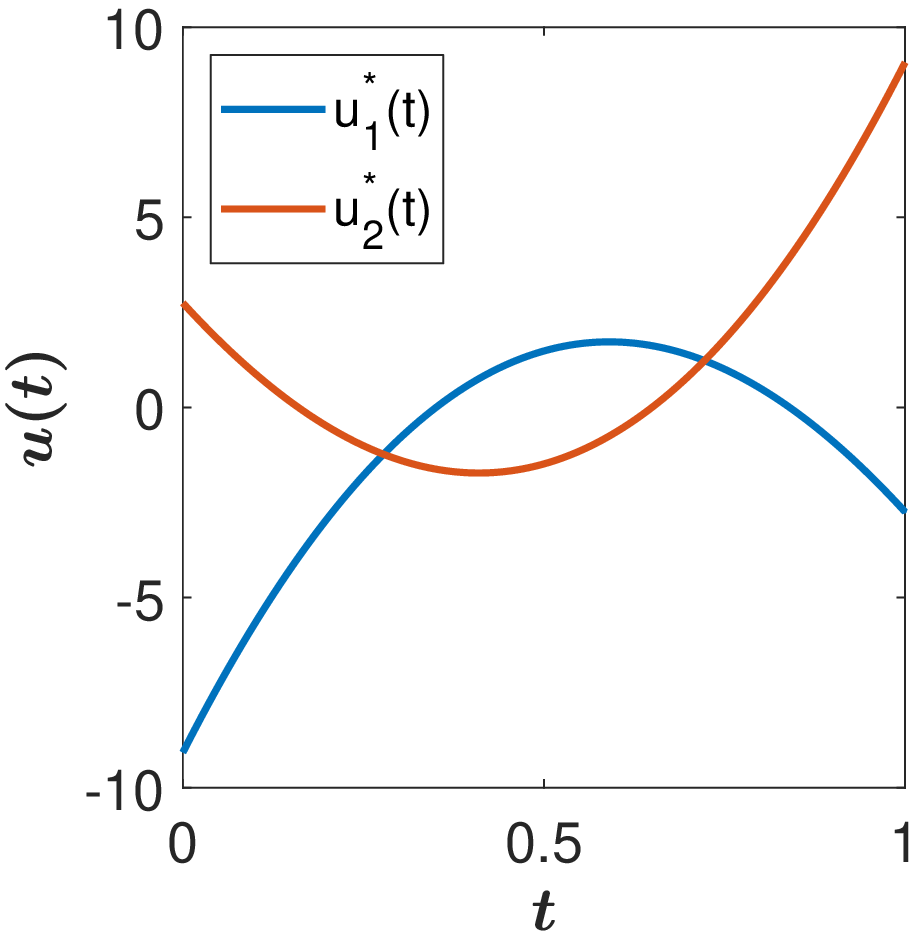}
  \end{subfigure}
  \hspace{5pt}
  \adjustbox{minipage=1em,valign=t}{\subcaption{}\label{fig: fixed end-point error}}%
  \begin{subfigure}[b]{0.45\linewidth}
    \centering\includegraphics[width=0.6\linewidth]{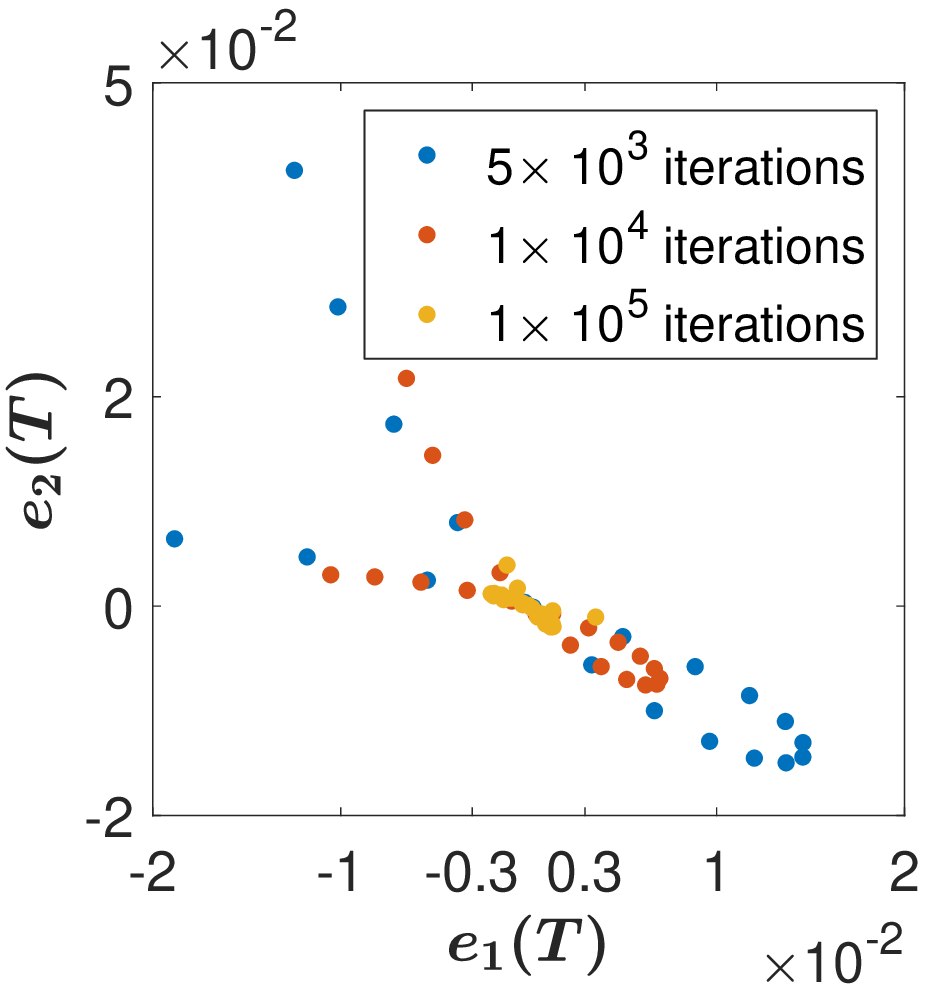} 
  \end{subfigure}
  \caption{Feasible control of the ensemble of 21 harmonic oscillators in \eqref{equ: harmonic oscillator}. (a) The feasible control law $(u_1^*(t), u_2^*(t))$ generated based on the update rule in \eqref{equ: update rule for linear ensemble} after $1\times 10^5$ iterations. (b) The terminal errors $(e_1(T), e_2(T))'$ of each harmonic oscillator under the feasible control laws obtained by applying the update rule in \eqref{equ: update rule for linear ensemble} for different numbers of iterations, where $(e_1(T), e_2(T))' = (x_1(T), x_2(T)-1)'$.}
\end{figure}

Figure \ref{fig: fixed end-point optimal control} reports the minimum-energy control law $(u_1^*(t),u_2^*(t))$ of Example \ref{ex: harmonic oscillator fixed endpoint} designed by the closed-form solution in \eqref{equ: closed form solution of weighted projection}. In this example, functions in $L^2([0, T], \mathbb{R}^2)$ are approximated using Legendre polynomials of order $r = 50$. The initial condition in \eqref{equ: closed form solution of weighted projection} is $\mu^{(0)} = 0$. Figure \ref{fig: fixed end-point optimal error} shows the terminal errors of each harmonic oscillator in the ensemble following the minimum-energy control law.

\begin{figure}[ht]
  \centering
  \adjustbox{minipage=1em,valign=t}{\subcaption{}\label{fig: fixed end-point optimal control}}%
  \begin{subfigure}[b]{0.45\linewidth}
    \centering\includegraphics[width=0.6\linewidth]{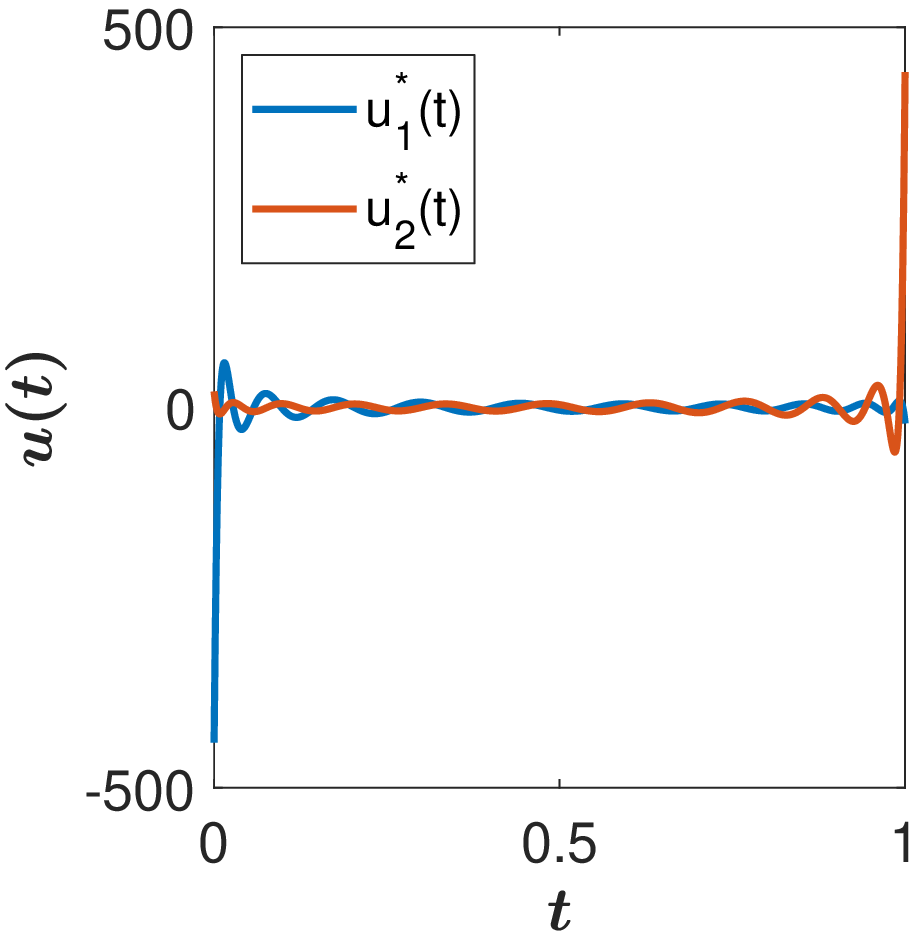}
  \end{subfigure}
  \hspace{5pt}
  \adjustbox{minipage=1em,valign=t}{\subcaption{}\label{fig: fixed end-point optimal error}}%
  \begin{subfigure}[b]{0.45\linewidth}
    \centering\includegraphics[width=0.6\linewidth]{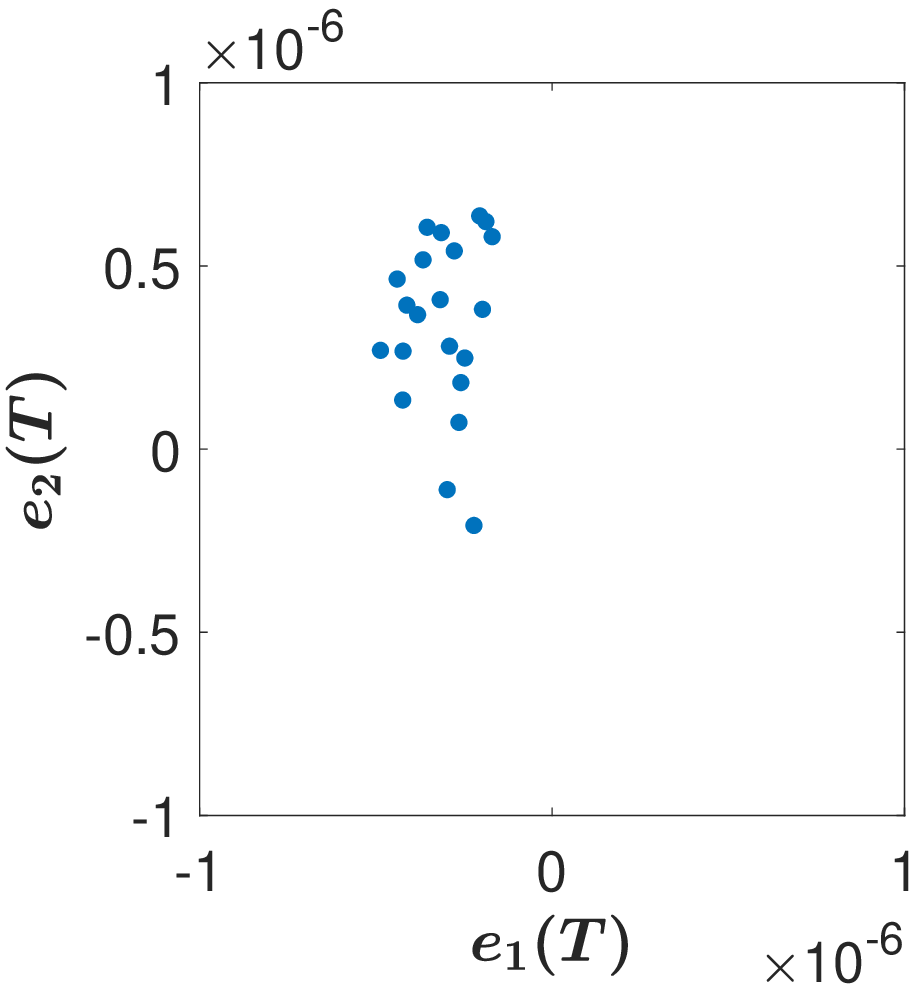} 
  \end{subfigure}
  \caption{Minimum-energy control of the ensemble of 21 harmonic oscillators in \eqref{equ: harmonic oscillator}. (a) The minimum-energy control law $(u_1^*(t), u_2^*(t))$ designed by the closed-form solution in \eqref{equ: closed form solution of weighted projection}. (b) The terminal errors $(e_1(T), e_2(T))'$ of each harmonic oscillator in the ensemble following the minimum-energy control, where $(e_1(T), e_2(T))' = (x_1(T), x_2(T)-1)'$.}
\end{figure}
\end{ex}

\begin{ex}[Pattern formation]\label{ex: shape control}
    In this example, we consider a more challenging task to drive an ensemble of $50$ harmonic oscillators in \eqref{equ: harmonic oscillator} between two configurations, where the initial state $X_0(\w_i)=(x_{10}(\w_i),x_{20}(\w_i))'$ and target state $X_F(\w_i)=(x_{1F}(\w_i),x_{2F}(\w_i))'$ at $T=40$, are arrangements of $50$ oscillators in star- and maple-shaped images in the plane as shown in Figure \ref{fig: shape_control_ini} and Figure \ref{fig: shape_control_final}, respectively, and $\omega_i$'s are uniformly sampled in $[-10, 10]$.

    Figure \ref{fig: shape_control_control} displays the minimum-energy control law $(u_1^*(t),u_2^*(t))$ designed by the closed-form solution in \eqref{equ: closed form solution of weighted projection}, which achieves the desired pattern formation. In this example, functions in $L^2([0, T], \mathbb{R}^2)$ are approximated using Legendre polynomials of order $r = 200$, and the initial condition in \eqref{equ: closed form solution of weighted projection} is $\mu^{(0)} = 0$.

\begin{figure}[ht]
  \centering
  \adjustbox{minipage=1em,valign=t}{\subcaption{}\label{fig: shape_control_ini}}%
  \begin{subfigure}[b]{0.43\linewidth}
    \centering\includegraphics[width=0.6\linewidth]{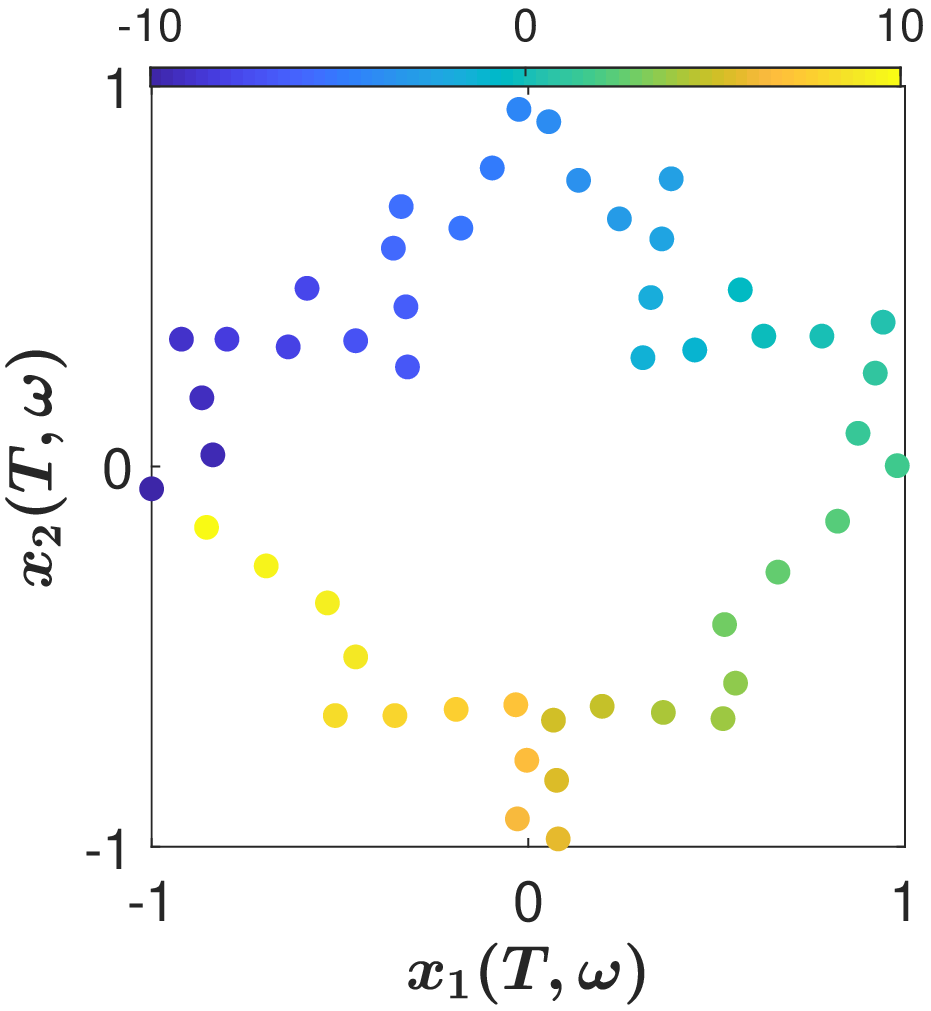}
  \end{subfigure}
  \hspace{20pt}
  \adjustbox{minipage=1em,valign=t}{\subcaption{}\label{fig: shape_control_final}}%
  \begin{subfigure}[b]{0.43\linewidth}
    \centering\includegraphics[width=0.6\linewidth]{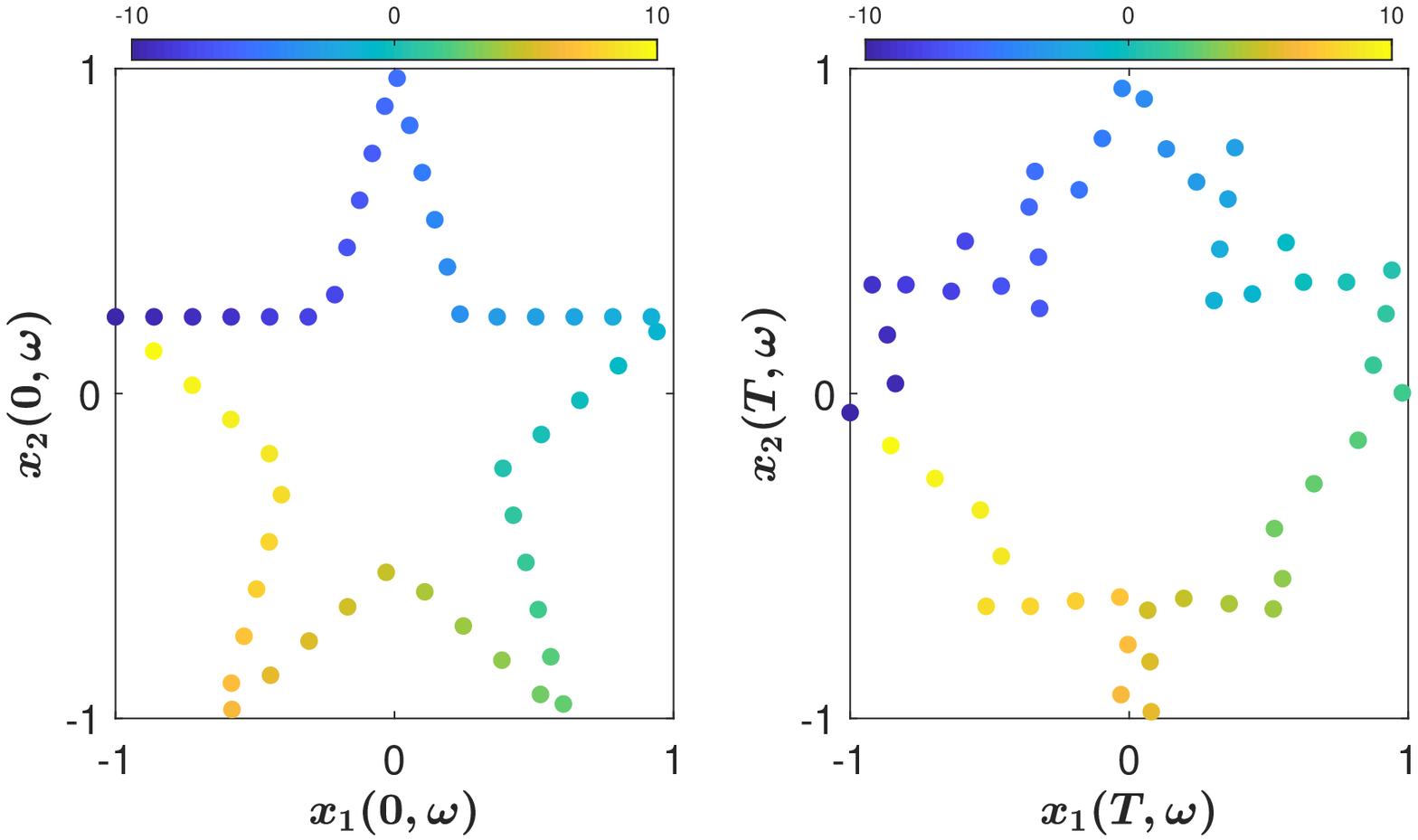}
  \end{subfigure}

  \adjustbox{minipage=1em,valign=t}{\subcaption{}\label{fig: shape_control_control}}%
  \begin{subfigure}[b]{1\linewidth}
    \centering\includegraphics[width=0.6\linewidth]{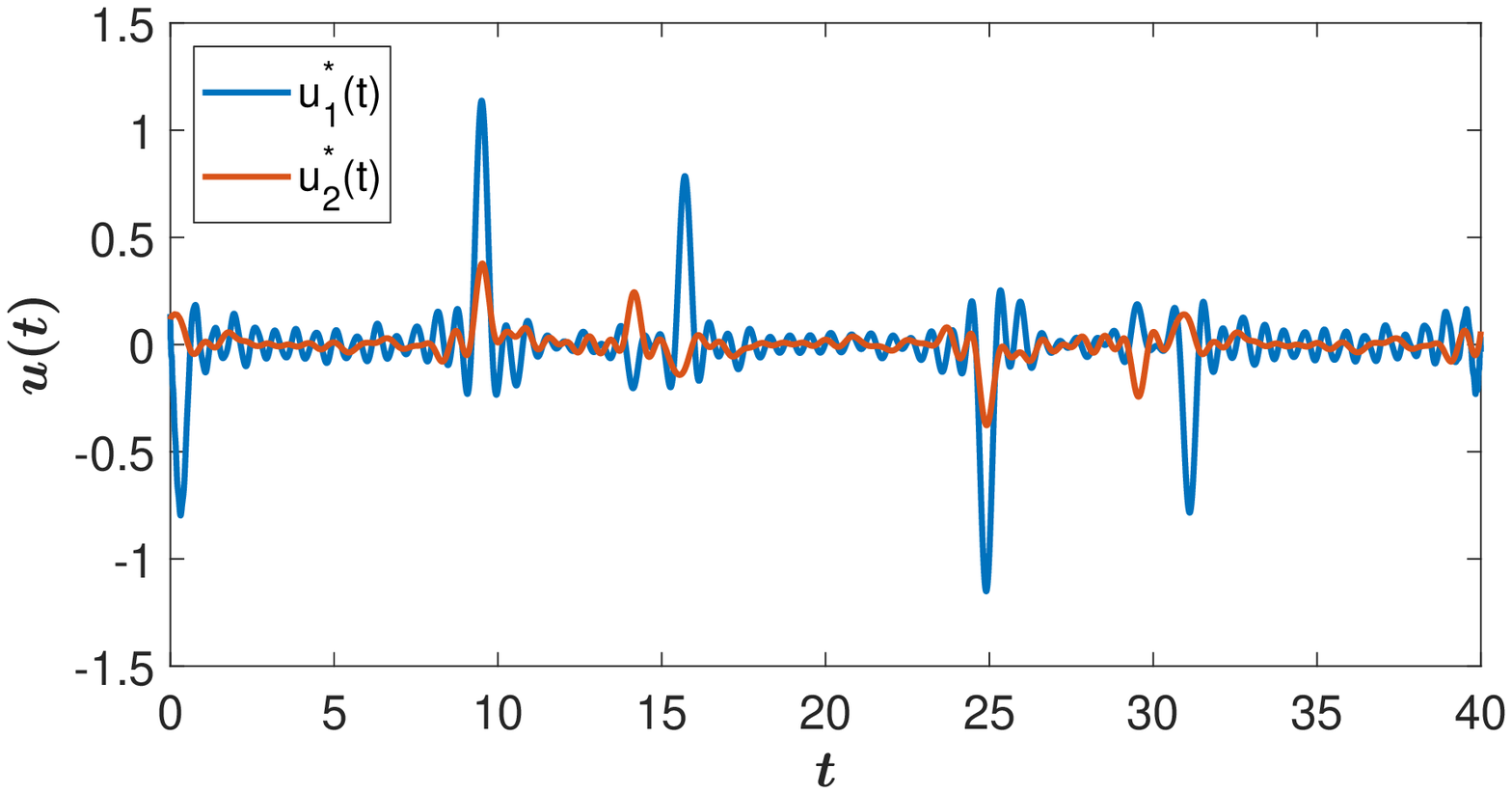} 
  \end{subfigure}
  \caption{Minimum-energy pattern formation of the ensemble of 50 harmonic oscillators in \eqref{equ: harmonic oscillator}. (a) Initial state $X_0(\w_i)$ of the ensemble. (b) Final state $(x_1(T, \w_i), x_2(T, \w_i))'$ of the ensemble, which coincide with the target state $X_F(\w_i)$. (c) The minimum-energy control that accomplishes the transfer between shapes.}
    \label{fig: shape_control}
\end{figure}
\end{ex}

\begin{ex}[Pattern formation in an uncontrollable ensemble]\label{ex: Shape control for uncontrollable ensemble}
  
    Here, we consider the same pattern formation problem as presented in Example \ref{ex: shape control}, while using only one control input $u(t)$, i.e.,
  \begin{align}
      \label{equ: harmonic oscillator uncontrollable}
      \frac{d}{dt} \begin{bmatrix}
          x_1(t, \omega_i) \\ x_2(t, \omega_i)
      \end{bmatrix} = \begin{bmatrix}
          0 & -\omega_i \\ \omega_i & 0
      \end{bmatrix} \begin{bmatrix}
          x_1(t, \omega_i) \\ x_2(t, \omega_i)
      \end{bmatrix} + \begin{bmatrix}
          1 \\ 0
      \end{bmatrix}u(t).
  \end{align}
  It was shown in \cite{li2016ensemble} that the ensemble system of the form
  \begin{equation*}
    \left\{\:\:
    \begin{aligned}
        &\frac{d}{dt} X(t, \beta) = \beta A X(t, \beta) + Bu(t),\\
        & \beta \in [-\beta_1, \beta_2], \quad \beta_1, \beta_2 >0,
    \end{aligned}\right.
  \end{equation*}
  with $A\in\mathbb{R}^{n\times n}$ and $B\in\mathbb{R}^{n\times m}$ is ensemble controllable if and only if $\text{rank}\:(A) = \text{rank}\:(B) = n$. As a result, the system in \eqref{equ: harmonic oscillator uncontrollable} is ensemble uncontrollable, i.e., the ensemble cannot be steered between an arbitrary pair of ensemble states, since the $B$ matrix is not full rank.

  This theoretical result indicates that the desired pattern formation from a star to a maple-shaped state may be unachievable. We will verify this by using the developed convex-geometric projection method. Given the desired pair of formation shown in Figures \ref{fig: shape_control_ini} and \ref{fig: shape_control_final}, we computed a feasible ensemble control law displayed in Figure \ref{fig: star_maple_control_uncontrollable}, using the closed-form solution in \eqref{equ: closed form solution of weighted projection}, for which the choices of the Legendre polynomials and the initial condition are the same as in Example \ref{ex: shape control}.

  The resulting final state following this ensemble control signal is shown in Figure \ref{fig: shape_control_final_uncontrollable}. Since it appears not to be of the desired shape of a maple leaf, by Theorem \ref{thm: feasible control and controllability for linear ensemble}, this final state is not ensemble reachable from the star-shaped state as shown in Figure \ref{fig: shape_control_ini}. This further implies that the system in \eqref{equ: harmonic oscillator uncontrollable} is ensemble uncontrollable.

\begin{figure}[ht]
  \centering
  \adjustbox{minipage=1em,valign=t}{\subcaption{}\label{fig: star_maple_control_uncontrollable}}%
  \begin{subfigure}[b]{0.43\linewidth}
    \centering\includegraphics[width = 0.6\linewidth]{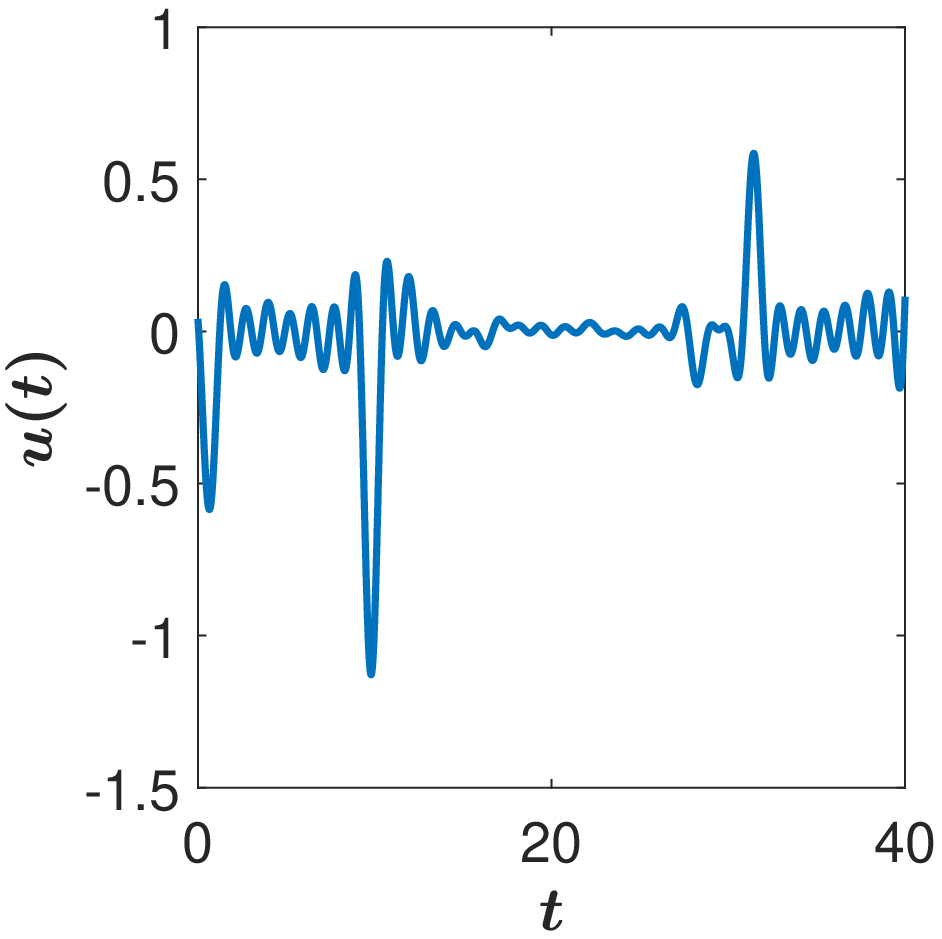} 
  \end{subfigure}
  \adjustbox{minipage=1em,valign=t}{\subcaption{}\label{fig: shape_control_final_uncontrollable}}%
  \begin{subfigure}[b]{0.43\linewidth}
    \centering\includegraphics[width=0.6\linewidth]{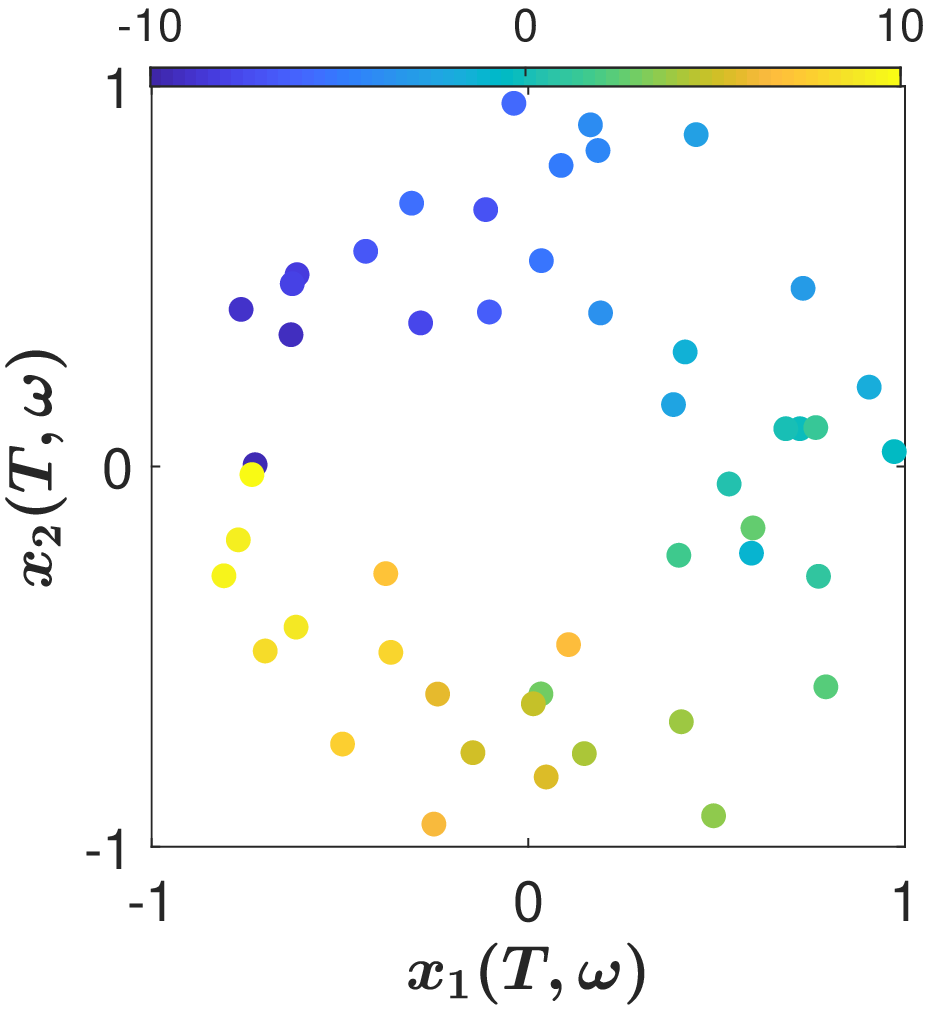}
  \end{subfigure}
  \caption{Results of the pattern formation for an uncontrollable ensemble. (a) The control law designed by the closed-form solution in \eqref{equ: closed form solution of weighted projection}. (b) Final state $(x_1(T, \w_i), x_2(T, \w_i))'$ of the ensemble. Since the final state is not identical to the target state, the desired maple-shaped configuration is not ensemble reachable.}
\end{figure}
\end{ex}

\subsubsection{Constrained linear ensemble control}
\label{sec:linear_example_constrained}
In the following two examples, we consider steering an ensemble of $21$ harmonic oscillators in \eqref{equ: harmonic oscillator} from the same initial state $(1, 0)'$ to the same target state $(0, 1)'$ at $T = 1$ with limited control energy or power, where $\omega_i$'s are uniformly sampled  in $[-10, 10]$.

\begin{ex}[Fixed-endpoint control with energy constraints]\label{ex: energy constrained control}
In this example, we consider energy constraints on the control inputs, given by 
$$G = \{(u_1, u_2)' \in L^2([0, T], \mathbb{R}^2) \:\big\vert\: \|u_1\|_2 \leq  M, \|u_2\|_2 \leq M\}.$$
\end{ex}

\begin{ex}[Fixed-endpoint control with amplitude constraints]\label{ex: amplitude constrained control}
In this example, we consider power constraints on the control input, given by 
$$G = \{(u_1,u_2)' \in L^2([0, T], \mathbb{R}^2) \:\big\vert\: \underset{t\in [0, T]}{\mathrm{max}}\{|u_1(t)|, |u_2(t)|\} \leq  M\}.$$
\end{ex}

Figures \ref{fig:C2_control_err} and \ref{fig:Cinf_control_err} display the results of Examples \ref{ex: energy constrained control} and \ref{ex: amplitude constrained control}, respectively, when $M = 5, 10, 25, 50$. In each case, the constrained control law is computed by applying the update rule in \eqref{equ: update rule for linear ensemble constrained} after $1\times 10^4$ iterations, where the initial control functions are $u_{1}^{(0)}(t) \equiv u_{2}^{(0)}(t) \equiv 0$. As we shall see from these figures, more relaxed constraints on the control inputs lead to smaller terminal errors.

\begin{figure*}[!t]
\centering
\includegraphics[width=6.5in]{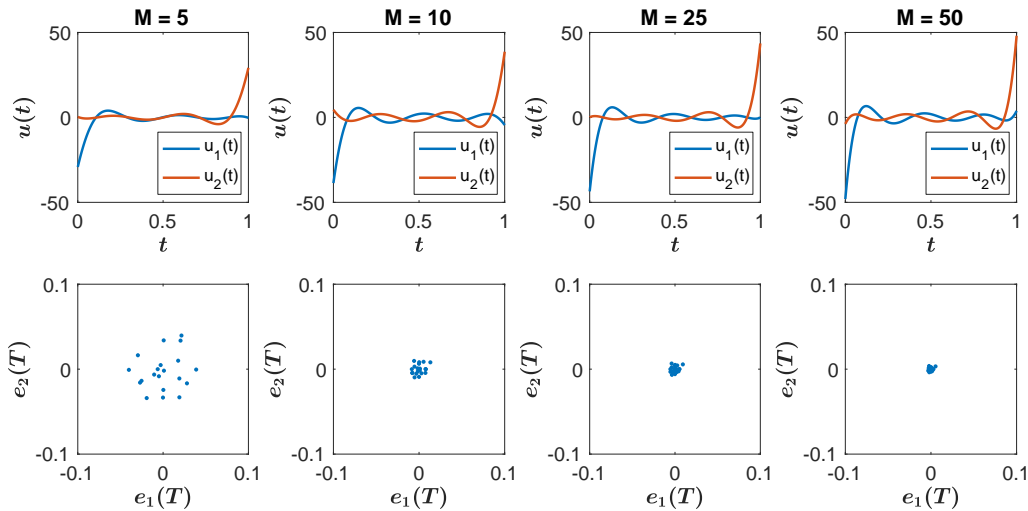}
\caption{Fixed-endpoint control with the constraints $\|u_1(t)\|_2\leq M$ and $\|u_2(t)\|_2 \leq M$ for $M = 5, 10, 25, 50$. The control law for each case is computed by running the update rule in \eqref{equ: update rule for linear ensemble constrained} for $1\times 10^4$ iterations, and $(e_1(T), e_2(T))'$ represents the terminal error computed by $(e_1(T), e_2(T))' = (x_1(T), x_2(T)-1)'$.}
\label{fig:C2_control_err}
\end{figure*}
\begin{figure*}[!t]
\centering
\includegraphics[width=6.5in]{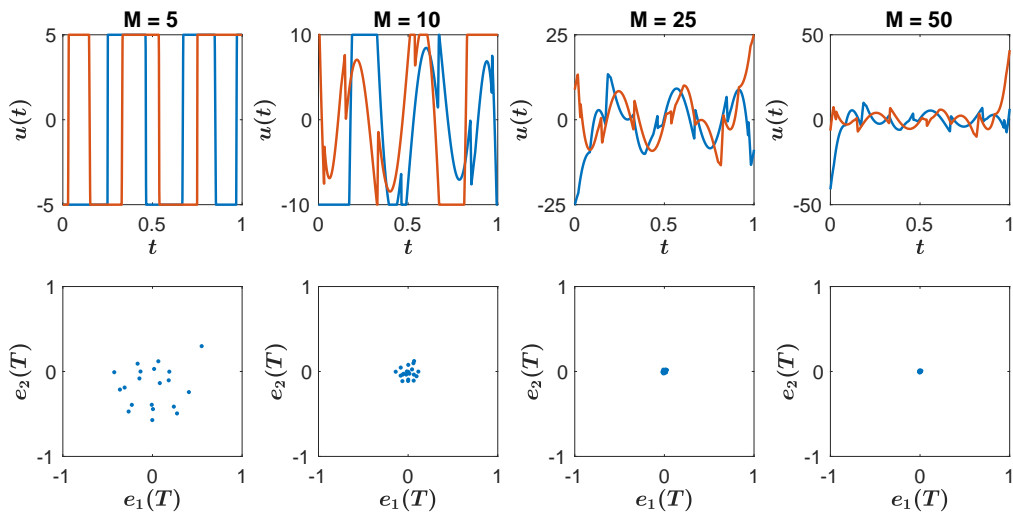}
\caption{Fixed-endpoint control with the constraints $\|u_1(t)\|_\infty \leq M$ and $\|u_2(t)\|_\infty  \leq M$ for $M = 5, 10, 25, 50$. The control law for each case is computed by running the update rule in \eqref{equ: update rule for linear ensemble constrained} for $1\times 10^4$ iterations, and $(e_1(T), e_2(T))'$ represents the terminal error computed by $(e_1(T), e_2(T))' = (x_1(T), x_2(T)-1)'$.}
\label{fig:Cinf_control_err}
\end{figure*}

\subsection{Bilinear Ensemble Systems}
\label{sec:bilinear_example}
In this section, we illustrate the application of the developed convex-geometric method to find feasible controls for bilinear ensemble systems as presented in Section \ref{sec: bilinear}. Our showcase is the control of a sample of nuclear spins modeled by the Bloch equations that form a bilinear ensemble system \cite{li2009ensemble}.

\begin{ex}[Broadband Pulse Design]\label{ex: bloch equations}
    The evolution of an ensemble of nuclear spins obeys the Bloch equations, given by
    \begin{align}
    \label{eq:Bloch}
      \frac{d}{dt} \begin{bmatrix}
      x(t, \omega_i) \\ y(t, \omega_i) \\ z(t, \omega_i)
      \end{bmatrix} = \begin{bmatrix}
      0 & -\omega_i & u(t) \\ \omega_i & 0 & -v(t) \\ -u(t) & v(t) & 0
      \end{bmatrix} \begin{bmatrix}
      x(t, \omega_i) \\ y(t, \omega_i) \\ z(t, \omega_i).
      \end{bmatrix},
    \end{align}
    where $\w_i\in [-1,1]$ denotes the Larmor frequency of the $i^{\text{th}}$ spin, $i=1,\ldots,41$, and $u(t)$ and $v(t)$ are radio-frequency fields (controls) applied to the $y$- and the $x$-axis, respectively. Here, we consider the design of a broadband inversion pulse \cite{li2011optimal} that uniformly inverts the ensemble in \eqref{eq:Bloch} from the identical initial state $(0,0,-1)'$ to the identical target state $(0, 0, 1)'$ at time $T = 1$, where $\omega_i$'s are uniformly sampled in $[-1, 1]$.
\end{ex}

To apply the integrated method involving the convex-geometric projection and the iterative procedure presented in Section \ref{sec: bilinear} to design a feasible control for the Bloch ensemble in \eqref{eq:Bloch}, we first rewrite its dynamics into a time-varying state-dependent linear ensemble form, that is,
\begin{align*}
    \quad\frac{d}{dt} \begin{bmatrix}
    x(t, \omega_i) \\ y(t, \omega_i) \\ z(t, \omega_i)
    \end{bmatrix} 
    &= \begin{bmatrix}
    0 & -\omega_i & 0 \\ \omega_i & 0 & 0 \\ 0 & 0 & 0
    \end{bmatrix} \begin{bmatrix}
    x(t, \omega_i) \\ y(t, \omega_i) \\ z(t, \omega_i)
    \end{bmatrix} + \begin{bmatrix}
     z(t, \omega_i) & 0 \\ 0 & -z(t, \omega_i)  \\ -x(t, \omega_i) & y(t, \omega_i) 
    \end{bmatrix} \begin{bmatrix}
    u(t)\\ v(t)
    \end{bmatrix}, \\
    &\doteq A(\omega_i) X(t, \omega_i) + \tilde{B}(X(t, \omega_i))U(t).
\end{align*}
Then, we consider the iteration equation, as introduced in \eqref{eq:iteration},
\begin{align}
    \frac{d}{dt}X^{(k+1)}(t, \omega_i)= & A(\omega_i)X^{(k+1)}(t, \omega_i) + \tilde{B}(X^{(k)}(t, \omega_i))U^{(k+1)}(t),
    \label{eq:iteration_ho}
\end{align}
where $k=0,1,2,\ldots$ denotes the iteration. Following the iterative method described in Section \ref{sec: bilinear} with the initial trajectory $X^{(0)}(t, \beta)$ generated by the initial control law $U^{(0)}(t) \equiv 0$, a convergent control-trajectory sequence, with the stopping criterion, $\mathrm{max}_{\omega_i} \|X(T, \omega_i) - X_F(\omega_i)\|_2 < 5\times 10^{-2}$, 
is generated after $300$ iterations, i.e., $k=0,1,\ldots,300$ in \eqref{eq:iteration_ho} and 
$$\{(U^{(k)}(t),X^{(k)}(t,\w_i))\}\to (U^{*}(t),X^{*}(t,\w_i)),$$ 
where $U^*(t)=(u^*(t), v^*(t))'$ is the convergent feasible ensemble control law displayed in Figure \ref{fig: bloch_control}. In each control-trajectory iteration, the minimum-energy ensemble control law for the time-varying linear ensemble in \eqref{eq:iteration_ho} is computed by applying the update rule in \eqref{equ: update rule for linear ensemble} for $1,000$ times, with the initial control functions $u^{(0)}(t) \equiv v^{(0)}(t) \equiv 0$. The resulting terminal errors and the trajectories for each Bloch equation following $(u^*(t), v^*(t))'$ are shown in Figure \ref{fig: bloch_error} and \ref{fig: bloch_traj}, respectively.

\begin{figure}[ht]
  \centering
  \adjustbox{minipage=1em,valign=t}{\subcaption{}\label{fig: bloch_control}}%
  \begin{subfigure}[b]{0.43\linewidth}
    \centering\includegraphics[width=0.6\linewidth]{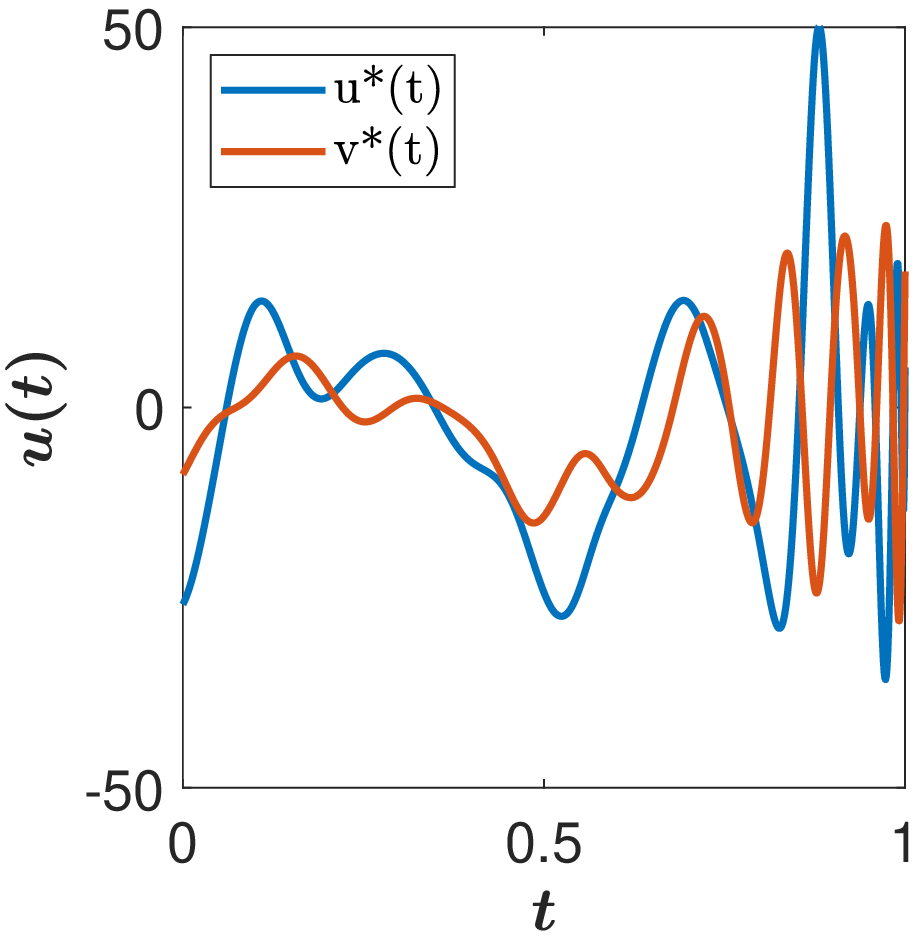}
  \end{subfigure}
  \hspace{0pt}
  \adjustbox{minipage=1em,valign=t}{\subcaption{}\label{fig: bloch_error}}%
  \begin{subfigure}[b]{0.43\linewidth}
    \centering\includegraphics[width=0.6\linewidth]{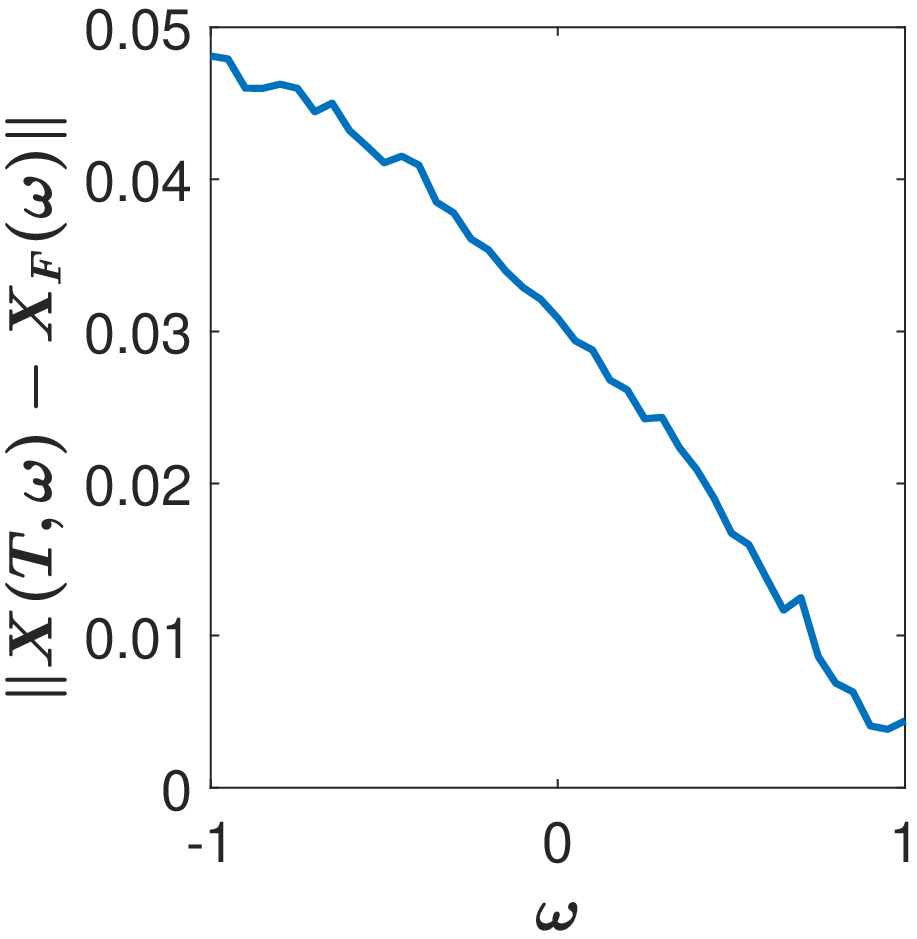}
  \end{subfigure}
  \hspace{0pt}
  \adjustbox{minipage=1em,valign=t}{\subcaption{}\label{fig: bloch_traj}}%
  \begin{subfigure}[b]{1\linewidth}
    \centering\includegraphics[width=0.5\linewidth]{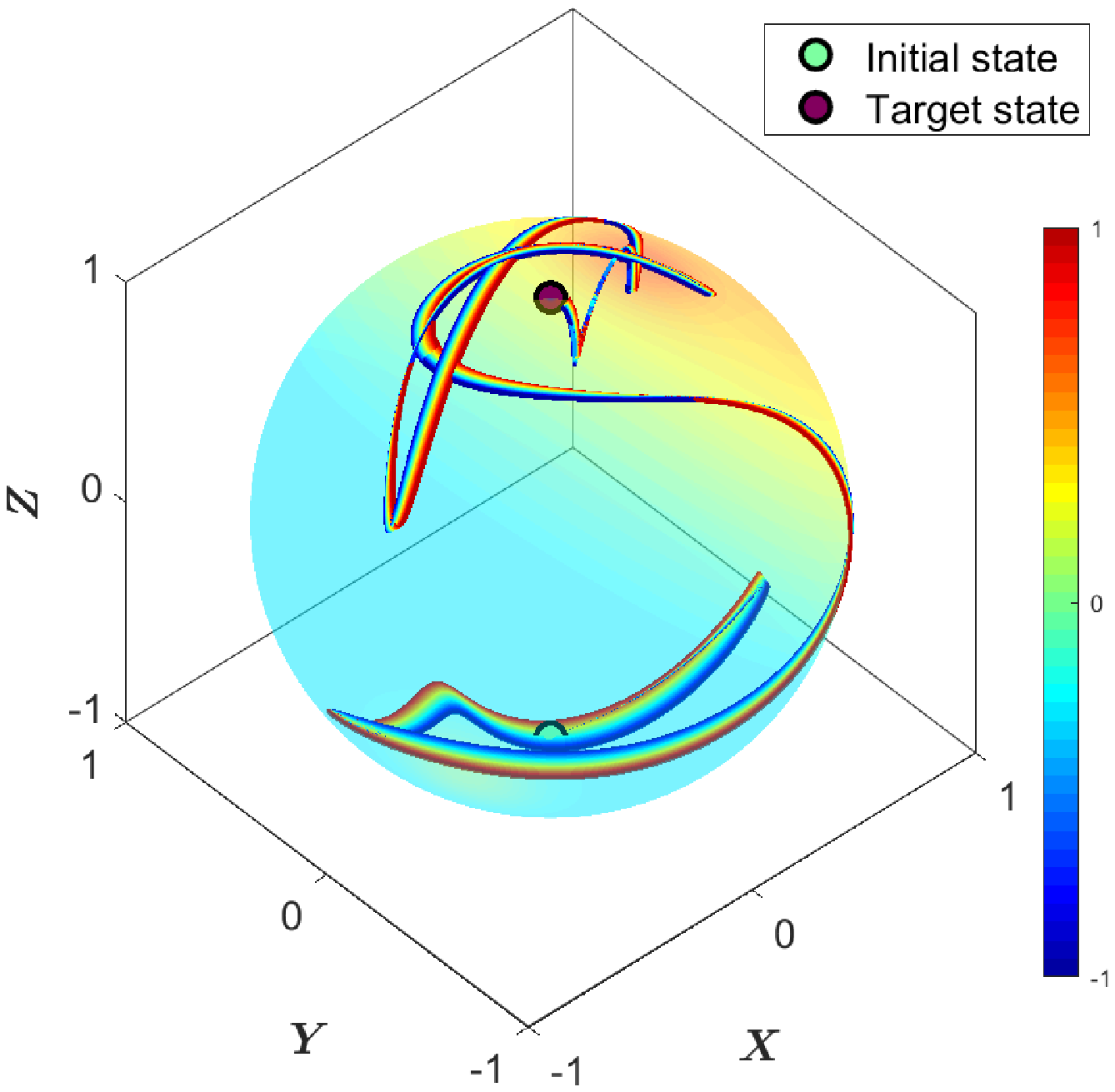}
  \end{subfigure}
  \caption{Feasible control of an ensemble of Bloch systems. (a) A feasible ensemble control law obtained after $300$ control-trajectory iterations as presented in \eqref{eq:sequence}. (b) and (c) illustrate the terminal error and the trajectory of each Bloch system by the application of the feasible ensemble control law in (a), respectively.}
\end{figure}

\section{Conclusion}
\label{sec: conclusion}

In this paper, we cast the problem of ensemble control design as a convex feasibility problem in a Hilbert space and proposed a convex-geometric approach for a systematic design of feasible and optimal controls for linear ensemble systems with or without constraints using the idea of iterative weighted projections. This approach also enabled a rigorous procedure for numerical evaluation of ensemble reachability and controllability for linear ensemble systems. This new addition expanded the numerical scope for understanding the fundamentals of ensemble control, which was mainly considered solely for control design. In addition to tackling linear ensemble systems, we proposed an iterated computational method combing the convex-geometric approach into an iterative framework for finding feasible control of bilinear ensemble systems. Numerical experiments were conducted to illustrate the applicability of the proposed convex-geometric approach to the design of ensemble control signals for linear and bilinear ensemble systems as well as to numerical verification of ensemble reachability and controllability.

\appendices

\section{Proof of the weighted projection algorithm}
\label{appendix: pf weighted projection}

Now we consider the following two sets in product space $\mathit{\Omega}$:
  \begin{align}
    \mathcal{C} &= C_1 \times \cdots \times C_N, \label{equ: pf def C}\\
    \mathcal{D} &= \{ U \in \mathit{\Omega} \:|\: u_1= \cdots =u_N \}. \label{equ: pf def D}
  \end{align}

Then $\bigcap_{i=1}^N C_i \neq \emptyset$ is equivalent to $\mathcal{C}\cap \mathcal{D}\neq \emptyset$.

Since $C_1, \ldots, C_N$ is closed and convex, $\mathcal{C}$ is also closed and convex. $\mathcal{D}$ is a subspace of $\mathit{\Omega}$, which is automatically closed and convex. Therefore the projections onto $\mathcal{C}$ and $\mathcal{D}$, denoted as $P_{\mathcal{C}}$ and $P_{\mathcal{D}}$, are well-defined. We associate each $u^{(k)}\in \mathcal{U}$ with $\tilde{U}^{(k)} := (u^{(k)}, \ldots , u^{(k)}) \in \mathit{\Omega}$. Then we prove following the update rule in \eqref{equ: weighted projection} to compute $\{u^{(k)}\}$, the associated sequence of $\{U^{(k)}\}$ satisfies
\begin{align*}
  \tilde{U}^{(k+1)} = P_{\mathcal{D}}P_{\mathcal{C}} \tilde{U}^{(k)}.
\end{align*}

We first show that $P_{\mathcal{C}} \tilde{U}^{(k)} = (P_{C_1}u^{(k)}, \ldots , P_{C_N} u^{(k)})$. By the definition of projection, we have $$P_{\mathcal{C}}\tilde{U}^{(k)} = \underset{\tilde{V}\in \mathcal{C}}{\text{argmin}} \| \tilde{U}^{(k)} - \tilde{V}\|_{\mathit{\Omega}} = \underset{\tilde{V}\in \mathcal{C}}{\text{argmin}} \| \tilde{U}^{(k)} - \tilde{V}\|^2_{\mathit{\Omega}}.$$

By the definition of inner product on $\mathit{\Omega}$, it holds that
\begin{align}
  &\|\tilde{U}^{(k)} - \tilde{V}\|_{\mathit{\Omega}}^2 = \langle \tilde{U}^{(k)} -\tilde{V}, \tilde{U}^{(k)} - \tilde{V} \rangle_{\mathit{\Omega}} \nonumber \\
  &= \sum_{i=1}^N \lambda_i \langle u^{(k)} - v_i, u^{(k)} - v_i \rangle_{\mathcal{U}} = \sum_{i=1}^N \lambda_i \|u^{(k)} - v_i\|^2_\mathcal{U}. \label{equ: pf 2}
\end{align}
Since each $v_i\in C_i \subset X$, by the definition of projection on $\mathcal{U}$, $\|u^{(k)} - v_i\|_\mathcal{U}^2 \geq \|u^{(k)} - P_{C_i}u^{(k)}\|_\mathcal{U}^2$. Hence in \eqref{equ: pf 2} can be bounded below by
\begin{equation} \label{equ: pf 3}
  \| \tilde{U}^{(k)} - \tilde{V}\|_{\mathit{\Omega}} \geq  \sum_{i=1}^N \lambda_i \| u^{(k)} - P_{C_i}u^{(k)}  \|_{\mathcal{U}}^2.
\end{equation}
We observe that $P_{C_i} u^{(k)} \in C_i$ for all $i$, the equality in \eqref{equ: pf 3} can be achieved when $\tilde{V} = (P_{C_1}u^{(k)}, \ldots , P_{C_N} u^{(k)})$. Hence we conclude that $P_{\mathcal{C}}\tilde{U}^{(k)} = (P_{C_1}u^{(k)}, \ldots , P_{C_N} u^{(k)})$.

Next, we show that $(u^{k+1}, \ldots, u^{(k+1)}) = P_{\mathcal{D}}P_{\mathcal{C}}\tilde{U}^{(k)}$.

Denote $\tilde{W} = (w, \ldots , w) = P_{\mathcal{D}}P_{\mathcal{C}}\tilde{U}^{(k)}$. We observe that $\mathcal{D}$ is a subspace in $\mathit{\Omega}$. Hence it holds that $\forall s\in \mathit{\Omega}, \forall t\in \mathcal{D}, s- P_\mathcal{D}s\perp t$. Now taking $s = P_{\mathcal{C}} \tilde{U}^{(k)}$ and $t = r\tilde{W}$, where $r$ is an arbitrary real number yields that
\begin{equation}\label{equ: pf 4}
  \langle P_\mathcal{C} \tilde{U}^{(k)} - \tilde{W}, r\tilde{W} \rangle_{\mathit{\Omega}} = 0.
\end{equation}
Substituting definition of inner product on $\mathit{\Omega}$ into in \eqref{equ: pf 4} yields that
\begin{equation}\label{equ: pf 5}
  \sum_{i=1}^N \lambda_i \langle P_{C_i} u^{(k)} - w, rw \rangle_\mathcal{U} = \langle \sum_{i=1}^N \lambda_i P_{C_i}u^{(k)} - w, rw \rangle_\mathcal{U}.
\end{equation}
Since in \eqref{equ: pf 5} holds for all $r\in \mathbb{R}$, $\sum_{i=1}^N \lambda_i P_{C_i} u^{(k)} - w$ must be $0$, which implies that $w = \sum_{i=1}^N \lambda_i P_{C_i} u^{(k)}$. By the update rule in \eqref{equ: weighted projection}, it holds that $u^{(k+1)} = \sum_{i=1}^N \lambda_i P_{C_i} u^{(k)}$. Hence we have $w = u^{(k+1)}$, which concludes that $\tilde{U}^{(k+1)} = W = P_\mathcal{D}P_\mathcal{C} \tilde{U}^{(k)}$.


\bibliographystyle{ieeetr}
\footnotesize
\bibliography{Ensemble_Projection}

\begin{thebibliography}{10}

\bibitem{glaser1998unitary}
S.~J. Glaser, T.~Schulte-Herbr{\"u}ggen, M.~Sieveking, O.~Schedletzky, N.~C.
  Nielsen, O.~W. S{\o}rensen, and C.~Griesinger, ``Unitary control in quantum
  ensembles: Maximizing signal intensity in coherent spectroscopy,'' {\em
  Science}, vol.~280, no.~5362, pp.~421--424, 1998.

\bibitem{li2006control}
J.-S. Li and N.~Khaneja, ``Control of inhomogeneous quantum ensembles,'' {\em
  Physical review A}, vol.~73, no.~3, p.~030302, 2006.

\bibitem{zeng2014tac}
S.~Zeng, S.~Waldherr, C.~Ebenbauer, and F.~Allg\"{o}wer, ``Ensemble
  observability of linear systems,'' {\em IEEE Transactions on Automatic
  Control}, vol.~61, no.~6, pp.~1452--1465, 2015.

\bibitem{li2011optimal}
J.-S. Li, J.~Ruths, T.-Y. Yu, H.~Arthanari, and G.~Wagner, ``Optimal pulse
  design in quantum control: A unified computational method,'' {\em Proceedings
  of the National Academy of Sciences}, 2011.

\bibitem{becker2012approximate}
A.~Becker and T.~Bretl, ``Approximate steering of a unicycle under bounded
  model perturbation using ensemble control,'' {\em IEEE Transactions on
  Robotics}, vol.~28, no.~3, pp.~580--591, 2012.

\bibitem{rosenblum2004controlling}
M.~G. Rosenblum and A.~S. Pikovsky, ``Controlling synchronization in an
  ensemble of globally coupled oscillators,'' {\em Physical Review Letters},
  vol.~92, no.~11, p.~114102, 2004.

\bibitem{Li_NatureComm16}
A.~Zlotnik, R.~Nagao, I.~Z. Kiss, and J.-S. Li, ``Phase-selective entrainment
  of nonlinear oscillator ensembles,'' {\em Nature Communications}, vol.~7,
  p.~10788, 2016.

\bibitem{brown2004phase}
E.~Brown, J.~Moehlis, and P.~Holmes, ``On the phase reduction and response
  dynamics of neural oscillator populations,'' {\em Neural Computation},
  vol.~16, no.~4, pp.~673--715, 2004.

\bibitem{li2013control}
J.-S. Li, I.~Dasanayake, and J.~Ruths, ``Control and synchronization of neuron
  ensembles,'' {\em IEEE Transactions on Automatic Control}, vol.~58, no.~8,
  pp.~1919--1930, 2013.

\bibitem{kafashan2015optimal}
M.~Kafashan and S.~Ching, ``Optimal stimulus scheduling for active estimation
  of evoked brain networks,'' {\em Journal of Neural Engineering}, vol.~12,
  no.~6, p.~066011, 2015.

\bibitem{li2009ensemble}
J.-S. Li and N.~Khaneja, ``Ensemble control of {Bloch} equations,'' {\em IEEE
  Transactions on Automatic Control}, vol.~54, no.~3, pp.~528--536, 2009.

\bibitem{zeng2016moment}
S.~Zeng and F.~Allgoewer, ``A moment-based approach to ensemble controllability
  of linear systems,'' {\em Systems \& Control Letters}, vol.~98, pp.~49--56,
  2016.

\bibitem{zeng2017sampled}
S.~Zeng, H.~Ishii, and F.~Allg{\"o}wer, ``Sampled observability and state
  estimation of linear discrete ensembles,'' {\em IEEE Transactions on
  Automatic Control}, vol.~62, no.~5, pp.~2406--2418, 2017.

\bibitem{chen2017distinguished}
X.~Chen and B.~Gharesifard, ``Distinguished vector fields over smooth manifolds
  with applications to ensemble control,'' in {\em 2017 IEEE 56th Annual
  Conference on Decision and Control (CDC)}, pp.~1963--1968, IEEE, 2017.

\bibitem{chen2019structure}
X.~Chen, ``Structure theory for ensemble controllability, observability, and
  duality,'' {\em Mathematics of Control, Signals, and Systems}, vol.~31,
  no.~2, p.~7, 2019.

\bibitem{schonlein2016controllability}
M.~Sch{\"o}nlein and U.~Helmke, ``Controllability of ensembles of linear
  dynamical systems,'' {\em Mathematics and Computers in Simulation}, vol.~125,
  pp.~3--14, 2016.

\bibitem{helmke2014uniform}
U.~Helmke and M.~Sch{\"o}nlein, ``Uniform ensemble controllability for
  one-parameter families of time-invariant linear systems,'' {\em Systems \&
  Control Letters}, vol.~71, pp.~69--77, 2014.

\bibitem{ruths2012optimal}
J.~Ruths and J.-S. Li, ``Optimal control of inhomogeneous ensembles,'' {\em
  IEEE Transactions on Automatic Control}, vol.~57, no.~8, pp.~2021--2032,
  2012.

\bibitem{phelps2014consistent}
C.~Phelps, Q.~Gong, J.~O. Royset, C.~Walton, and I.~Kaminer, ``Consistent
  approximation of a nonlinear optimal control problem with uncertain
  parameters,'' {\em Automatica}, vol.~50, no.~12, pp.~2987--2997, 2014.

\bibitem{Li_IFAC17}
L.~Tie, W.~Zhang, S.~Zeng, and J.-S. Li, ``Explicit input signal design for
  stable linear ensemble systems,'' {\em IFAC-PapersOnLine}, vol.~50, no.~1,
  pp.~3051 -- 3056, 2017.
\newblock 20th IFAC World Congress.

\bibitem{zeng2018computation}
S.~Zeng, W.~Zhang, and J.-S. Li, ``On the computation of control inputs for
  linear ensembles,'' in {\em 2018 Annual American Control Conference (ACC)},
  pp.~6101--6107, IEEE, 2018.

\bibitem{phelps2016optimal}
C.~Phelps, J.~O. Royset, and Q.~Gong, ``Optimal control of uncertain systems
  using sample average approximations,'' {\em SIAM Journal on Control and
  Optimization}, vol.~54, no.~1, pp.~1--29, 2016.

\bibitem{tie2017explicit}
L.~Tie, W.~Zhang, S.~Zeng, and J.-S. Li, ``Explicit input signal design for
  stable linear ensemble systems,'' {\em IFAC-PapersOnLine}, vol.~50, no.~1,
  pp.~3051--3056, 2017.

\bibitem{zlotnik2012synthesis}
A.~Zlotnik and S.~Li, ``Synthesis of optimal ensemble controls for linear
  systems using the singular value decomposition,'' in {\em American Control
  Conference (ACC), 2012}, pp.~5849--5854, IEEE, 2012.

\bibitem{Li_SICON17}
S.~Wang and J.-S. Li, ``Fixed-endpoint optimal control of bilinear ensemble
  systems,'' {\em SIAM Journal on Control and Optimization}, vol.~55, no.~5,
  pp.~3039--3065, 2017.

\bibitem{Li_Automatica18}
S.~Wang and J.-S. Li, ``Free-endpoint optimal control of inhomogeneous bilinear
  ensemble systems,'' {\em Automatica}, vol.~95, pp.~306--315, September 2018.

\bibitem{brockett2015finite}
R.~W. Brockett, {\em Finite Dimensional Linear Systems}, vol.~74.
\newblock SIAM, 2015.

\bibitem{von1950functional}
J.~Von~Neumann, {\em Functional Operators: Measures and Integrals}, vol.~1.
\newblock Princeton University Press, 1950.

\bibitem{bauschke1996projection}
H.~H. Bauschke and J.~M. Borwein, ``On projection algorithms for solving convex
  feasibility problems,'' {\em SIAM Review}, vol.~38, no.~3, pp.~367--426,
  1996.

\bibitem{boyle1986method}
J.~P. Boyle and R.~L. Dykstra, ``A method for finding projections onto the
  intersection of convex sets in hilbert spaces,'' in {\em Advances in Order
  Restricted Statistical Inference}, pp.~28--47, Springer, 1986.

\bibitem{li2016ensemble}
J.-S. {Li} and J.~{Qi}, ``Ensemble control of time-invariant linear systems
  with linear parameter variation,'' {\em IEEE Transactions on Automatic
  Control}, vol.~61, pp.~2808--2820, Oct 2016.

\end{thebibliography}

\end{document}